\documentclass{amsart}
\usepackage{amsmath, amssymb, euscript,  enumerate}

\newtheorem{thm}{Theorem}[section]
\newtheorem{cor}[thm]{Corollary}
\newtheorem{lem}[thm]{Lemma}
\newtheorem{claim}[thm]{Claim}
\newtheorem{remark}[thm]{Remark}
\newtheorem{prop}[thm]{Proposition}

\theoremstyle{remark}
\newtheorem{rem}[thm]{Remark}
\numberwithin{equation}{section}
\theoremstyle{plain}

\newcommand{\R}{\mathbb R}

\newcommand{\e}{\epsilon}

\newcommand{\laa}{\lambda}

\newcommand{\dv}{\mathrm{div}}

\newcommand{\tr}{\mathrm{trace}}

\newcommand{\ei}{\mathrm{ess \, inf}}
\newcommand{\diam}{\mathrm{diam}}
\newcommand{\ar}{\mathrm{area}}

\begin{document}

\title{The harmonic mean curvature flow of nonconvex surfaces in $\mathbb{R}^3$}

\author{Panagiota Daskalopoulos$^*$}

\address{Department of
Mathematics, Columbia University, New York,
 USA}
\email{pdaskalo@math.columbia.edu}

\author{Natasa Sesum$^*$}
\address{Department of Mathematics, Columbia University, New York,
USA}
\email{natasas@math.columbia.edu}

\thanks{$*:$ Partially supported
by NSF grant 0604657}

\begin{abstract}
We consider a compact  star-shaped mean convex hypersurface $\Sigma^2\subset \mathbb{R}^3$.
We prove that in some cases the flow exists until it shrinks to a point . We also prove that in the case of  a surface
of revolution which is star-shaped and mean convex, a smooth solution always exists up to some finite time $T < \infty$ at which the flow shrinks to a point asymptotically spherically.
\end{abstract}

\maketitle

\section{Introduction}

We will consider in this work the deformation of a compact
hyper-surface $\Sigma_t$ in $\R^3$ with no boundary under the {\em
  harmonic mean curvature flow} (HMCF) namely the flow
\begin{equation}\label{hmcf0}
\frac{\partial P}{\partial t} = - \frac{G}{H} \, \nu
\end{equation}
which evolves each point $P$ of the surface in the direction of its
normal unit vector with speed equal to the {\em harmonic mean
  curvature} of the surface $G/H$, with $G$ denoting the Gaussian
curvature of $\Sigma_t$ and $H$ its mean curvature.  Here $\nu$
denotes the outer unit normal to the surface at $P$.  This flow
remains weakly parabolic without the condition that $\Sigma_t$ is
strictly convex. However, it becomes degenerate at points where the
Gaussian curvature $G$ vanishes.

The existence of solutions to the HMCF with strictly convex smooth
initial data was first shown by Andrews in~\cite{An3} who also showed
that under the HMCF strictly convex  smooth surfaces converge to round
points in finite time. In~\cite{Di}, Di\"eter established the short
time existence of solutions to the HMCF with weakly convex smooth
initial data. More precisely, Di\"eter showed that if at time $t=0$
the surface $\Sigma_0$ satisfies $ G \geq 0$ and $H>0$, then there exists
a unique strictly convex smooth solution $\Sigma_t$ of the HMCF
defined on $0<t< \tau$, for some $\tau>0$. By the results of Andrews,
the solution will exist up to the time where its enclosed volume
becomes zero.

In \cite{CD} Caputo and the first author considered the highly
degenerate case where the initial surface is weakly convex with flat
sides, where the parabolic equation describing the motion of the
surface becomes highly degenerate at points where both curvatures $G$
and $H$ become zero. The solvability and optimal regularity of the
surface $\Sigma_t$, for $t >0$, was addressed and studied by viewing
the flow as a {\it free-boundary} problem.  It was shown that a
surface $\Sigma_0$ of class $C^{k,\gamma}$ with $k \geq 1$ and $0<\gamma
\leq 1$ at $t=0$,  will remain in the same class for $t >0$. In addition,
the strictly convex parts of the surface become instantly $C^\infty$
smooth up to the flat sides on $t >0$, and the boundaries of the flat
sides evolve by the curve shortening flow.

The case $G <0$ was recently studied by the first author and R.
Hamilton in \cite{DH1}, under the assumption that the initial surface
is a surface or revolution with boundary, and has $G <0$ and $H <0$
everywhere.  It was shown in \cite{DH1} that under certain boundary
conditions, there exists a time $T_0 >0$ for which the HMCF admits a
unique solution $\Sigma_t$ up to $T_0$, such that $H< 0$ for all $t<
T_0$ and $H(\cdot,T_0) \equiv 0$ on some set of sufficiently large
measure.  In addition, the boundary of the surface evolves by the
curve shortening flow.

\medskip

In this work we address the questions of short time and long time existence and
regularity of the HMCF under the assumption that $\Sigma_0$ is star-shaped
with $H >0$ but with $G$ changing sign.

Let $M^2$ be a smooth, compact surface without boundary and $F_0: M^2
\to \R^3$ be a smooth immersion of $M^2$. Let us consider a smooth family of
immersions $F(\cdot,t): M^2 \to \R^3$ satisfying
\begin{equation}
\label{equation-HMCF}
\frac{\partial F(p,t)}{\partial t} = -\kappa(p,t) \cdot   \nu (p,t) \tag{HMCF}
\end{equation}
where $\kappa = G/H$ denotes the harmonic mean curvature of $\Sigma_t
:= F(M^2,t)$ and $\nu$ its  outer unit normal at every point.  This is an equivalent
formulation of the HMCF.

For any compact two-dimensional surface $M^2$ which is smoothly
embedded in $\R^3$ by $F: M^2 \to \R^3$, let us denote by $g=(g_{ij})$
the induced metric, and by $\nabla$ the induced Levi-Civita
connection.  The second fundamental form $A = \{h_{ij}\}$ is a
symmetric bilinear form $A(p): T_p\Sigma \times T_p M \to \R$, defined
by $A(u,v) = \langle \nabla_u\nu, v\rangle$.  The Weingarten map
$W(p): T_p\Sigma \to T_p\Sigma$ of $T_p M$ given by the immersion $F$
with respect to the normal $\nu$, can be computed as $ h_j^i =
g^{ik}h_{kj}$. The eigenvalues of $W(p)$ are called the principal
curvatures of $F$ at $p$  and are denoted by $\lambda_1=\lambda_1(p)$
and $\lambda_2=\lambda_2(p)$. The mean curvature $H:= \tr(W) = \lambda_1 +
\lambda_2$, the total curvature $|A|^2 := \tr (W^t\, W) = \lambda_1^2
+ \lambda_2^2$ and the Gauss curvature $G = \det W = \lambda_1\, 
\lambda_2$.

\medskip



\begin{rem}
\label{rem-andrews}
We will recall some standard facts  about homogeneous of degree one functions
of matrices that  can be found in \cite{An1}. The speed  speed $\kappa$ of the interface evolving by
the HMCF  can be viewed
as a  function of the Weingarten map $W$ and therefore,  more generally,
as a function $\kappa: S \to \R$, where $S$ denotes  the set of all symmetric, 
positive transformations of $T\Sigma^2$  with strictly positive trace. 
Let $\lambda_1, \lambda_2$ be the eigenvalues of $A \in S$. 
We can then  define the  symmetric function $f(\lambda_1,\lambda_2) := \kappa(A)$.
We have:
\begin{itemize}
\item
If $f$ is concave (convex) and $\lambda_i > \lambda_j$, then
$\frac{\partial f}{\partial \lambda_i} - \frac{\partial f}{\partial \lambda_j}$
is negative (positive).
\item
Let $\ddot{\kappa} \in T\Sigma \otimes T^*\Sigma \otimes T\Sigma \otimes T^*\Sigma$
denote the second derivative of $\kappa$ at the point $A \in S$. If $A$ is
diagonal, then
\begin{equation}
\label{eq-ddot} 
\ddot{\kappa}(\xi,\eta) = \sum_{p,q}\frac{\partial^2\kappa}{\partial\lambda_p\lambda_q}
\xi_p^p\eta_q^q + \sum_{p\neq q} \frac{\frac{\partial\kappa}{\partial\lambda_p}
- \frac{\partial\kappa}{\partial\lambda_q}}{\lambda_p - \lambda_q}\xi_p^q\eta_p^q.
\end{equation}
\end{itemize}
\end{rem}

\bigskip

The  outline of the paper is as follows: 
\begin{enumerate}[i.]
\item In section \ref{section-ste} we will establish the short time
existence of (HMCF),  under the assumption that the initial surface $\Sigma_0$ is
compact  of class $C^{2,1}$ and has $H >0$. To do so we will have to bound $H$ from below away from zero independently of $\e$. 
This does not follow naturally from the evolution of $H$. To obtain such a bound 
we need to combine the evolution of $H$ with the evolution of the gradient of
the second fundamental form. This explains our assumption that $\Sigma \in C^{2,1}$. 

\item In section \ref{section-ltee} we will study 
the long time existence of the regularized flow (HMCF$_\e$) (defined in the next section).
We will show that  there exists a maximal time 
of existence $T_\e$  of a smooth solution  $\Sigma^{\e}_t$ of (HMCF$_\e$) such that
either  $H(P_t,t) \to 0$, as $t \to T_\e$  at some points $P_t  \in \Sigma_t^\e$, or
 $\Sigma^{\e}_t$ shrinks  to a point as
$t \to T_{\e}$. In addition,  we will establish uniform in $\e$ curvature bounds and curvature pinching
estimates.  
In the special case where the   initial data is a  surface of revolution, we will show that 
the flow always exists up to the time when the surface shrinks to a point. 

\item In section \ref{section-lte} we will pass to the limit,  $\e \to 0$,  to establish the long time existence of (HMCF). 
\end{enumerate}

\section{Short time Existence}\label{section-ste}

Our goal in this section is to show the following short time existence result for
the  HMCF. 

\begin{thm}
\label{thm-STE}
Let $\Sigma_0$ be a compact    hyper-surface in $\R^3$ which is of class $C^{2,1}$  and
has strictly positive mean curvature  $H > 0$.
Then, there  exists  $T > 0$  for which the harmonic mean curvature flow (\ref{equation-HMCF}) 
admits a unique  $C^{2,1}$ solution $\Sigma_t$, such that $H > 0$ on  $t\in [0,T)$.
\end{thm}

\medskip
Because  the harmonic mean curvature flow becomes   degenerate when the Gauss curvature of the surface $\Sigma_t$ changes sign, we will show 
the short time existence for equation  (\ref{equation-HMCF}) by considering 
its  $\epsilon$-regularization of the flow  defined by 
\begin{equation}
\label{equation-reg}\tag{HMCF$_\e$}
\frac{\partial F_{\epsilon}}{\partial t} = -(\frac{G}{H} + \epsilon \, H)\cdot   \nu 
\end{equation}
and starting at $\Sigma_0$. We  will denote by $\Sigma_t^\e$ the surfaces obtained by
evolving the initial surface $\Sigma_0$ along the flow \eqref{equation-reg}.

Since the right hand side of \eqref{equation-reg} can be viewed as a
function of the second fundamental form matrix $A$, a direct
computation shows that its linearization is given by
$$\mathcal{L}_{\epsilon}(u) = \frac{\partial}{\partial h_k^i}
\left (\frac{G}{H} + \epsilon H \right )\nabla_i \, \nabla_k u = a_{\epsilon}^{ik}\, \nabla_i \, \nabla_k u$$
with 
\begin{equation}
\label{eq-coeff}
a_{\epsilon}^{ik} =  \frac{\partial}{\partial h_k^i}\left (\frac{G}{H} + \epsilon H \right).
\end{equation}   
Notice that if we compute $a^{ik}_\e$ in geodesic coordinates around
the point (at which the matrix $A$ is diagonal) we get
\begin{equation}
\label{eq-coeff1}
a^{ik}_\epsilon = \left( \begin{array}{cc}
\frac{\lambda_2^2}{(\lambda_1 + \lambda_2)^2} + \epsilon & 0 \\
0 & \frac{\lambda_1^2}{(\lambda_1 + \lambda_2)^2} + \epsilon \\
\end{array} \right)
\end{equation} 
which is strictly positive definite, no matter what the principal curvatures are. 

The following  short time existence for the regularized flow \eqref{equation-reg}
follows from the standard theory on the existence of solutions to
strictly parabolic equations.

\begin{prop}\label{prop-ste-e}Let $\Sigma_0$ be a compact  hyper-surface in $\R^3$ which is of class $C^{1,1}$  and
has strictly positive mean curvature  $H > 0$.
Then, there  exists  $T_\e  > 0$,  for which the harmonic mean curvature flow (\ref{equation-reg}) 
admits a smooth  solution $\Sigma_t^\e$, such that $H > 0$ on  $t\in [0,T_\e)$. 
\end{prop}

Our goal is to show that if the initial surface $\Sigma_0$ is  of class 
$C^{2,1}$, then there is a $T_0> 0$,  so that $T_{\epsilon} \ge T_0$
for every $\epsilon$ and that we have uniform estimates on
$F_{\epsilon}$, independent of $\epsilon$, so that we can take a limit
of $F_{\epsilon}$ as $\epsilon\to 0$ and obtain  a solution of
(\ref{equation-HMCF}) that is of class $C^{2,1}$. The main obstacle here is to exclude that  $H_\e(P_\e,t_\e) \to 0$,
as $\e \to 0$,  for some points $P_\e \in \Sigma_{t_\e}^\e$ and times $t_\e \to 0$. Notice that our flow
cannot be defined at points where $H=0$. 
\bigskip

\noindent{\em Notation.}   \begin{enumerate}[$\bullet$]

\item When there is no
possibility of confusion,  we will use the letters $c$, $C$ and $T_0$ 
for various  constants  which are independent of $\e$ 
but change from line to line.

\item  Throughout this section we will  denote by $\lambda_1, \lambda_2$ the two principal curvatures of the surface $\Sigma_t^\e$
at a point $P$ and will assume  that
$\lambda_1 \ge \lambda_2.$  

\item  When there is no possibility of confusion we
will drop the index $\e$ from $H, G, A, g_{ij}, h_{ij}$ etc.  

\end{enumerate}

\medskip

The next lemma  follows directly from the computations of B. Andrews in \cite{An1} (Chapter 3). 

\begin{lem}
\label{lemma-evolution}
If $\Sigma^{\epsilon}_t$ moves by (\ref{equation-reg}), with speed
$\kappa_\epsilon := \frac{G}{H} + \epsilon H$, the computation in
\cite{An1} gives us the evolution equations
\begin{enumerate}[i.]
\item 
$\frac{\partial}{\partial t}H = \mathcal{L}_{\epsilon} H +
\frac{\partial^2 \kappa_\epsilon}{\partial h_q^p\partial h_m^l}
\nabla^i h_q^p\nabla_j h_m^l + \frac{\partial
\kappa_\epsilon}{\partial h_m^l}h_p^lh_m^p\,  H$
\item 
$\frac{\partial}{\partial t} \kappa_\epsilon =
\mathcal{L}_{\epsilon} \kappa_\e + \frac{\partial \kappa_\epsilon}{\partial
h_j^i}h_{il}h_{lj}\, \kappa_\epsilon$.
\end{enumerate}
\end{lem}

Note that if $\kappa := \frac{G}{H}$, we have 
\begin{equation}
\label{eq-useful0}
\frac{\partial\kappa}{\partial h_p^q}h_m^qh_p^m = \sum_{i=1}^2\frac{\partial\kappa}
{\partial\lambda_i}\lambda_i^2 = 2\, \kappa^2 \end{equation}
hence
\begin{equation}
\label{eq-useful00}
\frac{\partial\kappa_\e}{\partial h_p^q}h_m^qh_p^m = \sum_{i=1}^2\frac{\partial\kappa_\e}
{\partial\lambda_i}\lambda_i^2 = 2\, \kappa^2 + \e\, |A|^2
\end{equation}
with $|A|^2 = \lambda_1^2+\lambda_2^2$. 
We then conclude from the above lemma that $H$ and $\kappa_\e$ satisfy the evolution equations

\begin{equation}\label{eqn-H}
\frac{\partial}{\partial t}H = \mathcal{L}_{\epsilon} H +
\frac{\partial^2 \kappa_\epsilon}{\partial h_q^p\partial h_m^l}
\nabla^i h_q^p\nabla_j h_m^l + (2\, \kappa^2 + \e\, |A|^2)\,  H
\end{equation}
and
\begin{equation}\label{eqn-k}
\frac{\partial \kappa_\epsilon^2}{\partial t}  =
\mathcal{L}_{\epsilon} \kappa_\e +  (2\, \kappa^2 + \e\, |A|^2)\, \kappa_\epsilon^2.
\end{equation}

\medskip

We will now combine the above evolution equations to establish the  following uniform bound on the second fundamental form.

\begin{prop}
\label{prop-H}
There  exist  uniform constants $C$ and $T_0$ so that 
\begin{equation}
\label{eqn-A}
\max_{\Sigma_t^{\e}}|A| \le C, \qquad  \forall t\in [0,\min (T_{\epsilon}, T_0)\, ).
\end{equation}
\end{prop}

\begin{proof}
Recall that $H$ satisfies  the equation \eqref{eqn-H}. If we multiply this equation by $H$, we get
$$\frac{\partial H^2}{\partial t} = \mathcal{L}_{\epsilon}(H^2) - 2a_{\epsilon}^{ik}\nabla_i H\nabla_k H
+ \frac{\partial^2 \kappa_\e}{\partial h_q^p\partial h_m^l}
\nabla^i h_q^p\nabla_i h_m^l\cdot H + 2 (2 \kappa^2 + \e \, |A|^2)\,   H^2$$
with $\kappa=G/H$. 
Notice that the  definiteness of the matrix $D^2 \kappa_\e = [\frac{\partial^2 \kappa_\e}
{\partial h_p^q\partial h_m^l}]$ depends on the sign of $H$. It is easier to check this  in geodesic
coordinates around a point at which the Weingarten map is diagonalized. In those coordinates, by
(\ref{eq-ddot}), we have
$$\sum_{p,q,m,l}\frac{\partial^2 \kappa_\e}{\partial h_q^p\partial h_m^l}\nabla^i h_q^p\nabla_i h_m^l
= \sum_{p,q}\frac{\partial^2\kappa_{\e}}{\partial\lambda_p\lambda_q}\nabla^i h_p^p\nabla_i h_q^q
+ \sum_{p\neq q}\frac{\frac{\partial\kappa_{\e}}{\partial\lambda_p} - 
\frac{\partial\kappa_{\e}}{\partial\lambda_q}}{\lambda_p - \lambda_q}(\nabla_i h_p^q)^2$$
where the matrix $D^2\kappa_{\e} := [\frac{\partial^2\kappa_{\e}}{\partial\lambda_p\lambda_q}]$
is given by 
\begin{equation}
\label{eq-concavity}
 D^2\kappa_\e = \left(\begin{array}{cc}
-\frac{2\lambda_2^2}{(\lambda_1 + \lambda_2)^3}  & \frac{2\lambda_1\lambda_2}{(\lambda_1+\lambda_2)^3} \\
\frac{2\lambda_1\lambda_2}{(\lambda_1+\lambda_2)^3} & -\frac{2\lambda_1^2}{(\lambda_1 + \lambda_2)^3}  \\
\end{array} \right) = -\frac{2}{H^3}\left(\begin{array}{cc}
\lambda_2^2 & -\lambda_1\lambda_2 \\
-\lambda_1\lambda_2 & \lambda_1^2\\ \end{array} \right)
\end{equation} 
and for $p\neq q$,
$$\frac{\frac{\partial\kappa_{\e}}{\partial\lambda_p} - 
\frac{\partial\kappa_{\e}}{\partial\lambda_q}}{\lambda_p - \lambda_q}
= \frac{\lambda_q^2 - \lambda_p^2}{\lambda_p - \lambda_q} = -\frac{1}{H}.$$
It is now easy to see that
\begin{equation}
\label{eq-definite-H}
\frac{\partial^2 \kappa_\e}{\partial h_q^p\partial h_m^l}
\nabla^i h_q^p\nabla_i h_m^l\cdot H \le 0
\end{equation}
hence 
\begin{equation}\label{eqn-H2}
\frac{\partial H^2}{\partial t} \le \mathcal{L}_{\epsilon}(H^2) + 2\, (2  \kappa^2 + \e \, |A|^2)\,   H^2.
\end{equation}
Similarly, from the evolution of $\kappa_\e$, namely \eqref{eqn-k},  we obtain
\begin{equation}\label{eqn-k2}
\frac{\partial \kappa_\epsilon^2}{\partial t}  \leq
\mathcal{L}_{\epsilon} (\kappa_\e^2)  + 2\,  (2 \kappa^2 + \e\, |A|^2)\, \kappa_\epsilon^2. 
\end{equation}
We observe that because of   the appearance of the second fundamental form  $|A|^2$ in  the
zero order term  of   the  equations \eqref{eqn-H2} and \eqref{eqn-k2}, we cannot estimate the maximum 
of $H^2$ and $\kappa_\e^2$ directly from each equation using the maximum principle. This is because the
surface is not convex. However, it is possible to estimate the maximum of $H^2+\kappa_\e^2$ by combining the
two evolution equations. 
To this end, we set $M=H^2 + \kappa_\e^2$ and compute,  by adding the last two equations,  that
\begin{equation}\label{eqn-M}
\frac{\partial M}{\partial t}  \leq
\mathcal{L}_{\epsilon}  M   +  2\,  (2 \kappa^2 + \e\, |A|^2) \,  M. 
\end{equation}
We will  show the bound
\begin{equation}\label{eqn-bound}
2 \kappa^2 + \e\, |A|^2 \leq C\, (H^2 + \kappa_\e^2)
\end{equation}
for some uniform in $\e$ constant $C$, 
where $\kappa_\e = \kappa+ \e\, H$ and $\kappa = G/H$. 
Since $\kappa \leq \kappa_\e$ it will be sufficient to show that
\begin{equation}\label{eqn-AA}
|A|^2 \leq C\, (H^2 + \kappa^2).
\end{equation}
Expressing everything in terms of the principal curvatures $\lambda_1$ and
$\lambda_2$, the above reduces to the estimate
$$\laa_1^2 + \laa_2^2 \leq C\, \left ( (\laa_1+\laa_2)^2 + \left ( \frac{\laa_1\, \laa_2}{\laa_1+\laa_2} \right )^2 \right ).$$
If $\laa_2=0$ the above inequality is clearly satisfied. 
Assume that $\laa_2 \leq \laa_1$ with $\laa_2 \neq 0$  and set $\mu = {\laa_1}/{\laa_2}$. Since $H=\laa_1+\laa_2 >0$, we conclude that $|\mu| \geq 1$. 
Then, the last inequality is expressed as 
$$1+ \frac 1{\mu^2} \leq C \,  \left ( \frac{(1+\mu)^2}{\mu^2} + \frac{1}{(1+\mu)^2} \right )
$$
which can be reduced to showing that
$$ \frac{(1+\mu)^2}{\mu^2} + \frac{1}{(1+\mu)^2} \geq c >0$$
for a uniform constant $c$, since $|\mu| \geq 1$. This inequality is clearly satisfied
when $|\mu| \geq 1$. Hence, \eqref{eqn-bound} holds.

Applying \eqref{eqn-bound} on \eqref{eqn-M} we conclude get
$$\frac{\partial M}{\partial t}  \leq
\mathcal{L}_{\epsilon}  M   +  \theta \,   M^2$$
for some uniform constant $\theta$. The maximum principle then implies  the
differential inequality
$$\frac{d M_{\max}}{d t}  \leq  \theta \,   M_{\max}^2$$
which readily implies that 
$$\max_{\Sigma_t^{\e}} M \le C, \qquad  \forall t\in [0,\min (T_{\epsilon}, T_0)\, )$$
for some  uniform in $\e$  constants $C$ and $T_0$. This combined with
\eqref{eqn-AA} implies \eqref{eqn-A} finishing the proof of the proposition. 
\end{proof}

To establish the short existence of the flow \eqref{equation-HMCF} on $(0,T_0)$ for some
$T_0 >0$,  we still need to bound $H$ from below away from zero independently of $\e$. 
This does not follow naturally from the evolution of $H$, because the equation 
\eqref{eqn-H} carries   a quadratic negative term which depends
on the derivatives of the second fundamental form. Hence to establish the lower bound on $H$ 
we need to combine the evolution of $H$ with the evolution of the gradient of
the second fundamental form. This is shown in the next proposition.

\begin{prop}
\label{prop-nabla-curv}
There exist uniform in $\e$ positive constants $T_0 $,
$C$ and $\delta$,  so that 
$$|\nabla A| \le C \quad \mbox{and} \quad  H \ge \delta, 
\quad \mbox{on}\,\,  \Sigma_t^{\e}$$
 for $t\in [0,
\min (T_\e,T_0)\,)$.
\end{prop}

\begin{proof}
We will first compute the evolution equation for $\sum_{i,j}|\nabla h_i^j|^2$. Lets first
see how $h_i^j$ evolves. We have 
$$\frac{\partial}{\partial t}h_i^j = \mathcal{L}_{\e}(h_i^j) + 
\frac{\partial^2 \kappa_{\e}}{\partial h_q^p\partial h_m^l}
\nabla^i h_q^p \nabla^j h_m^l + \frac{\partial \kappa_{\e}}{\partial h_m^l}h_p^l h_m^p h_j^i.$$
>From the previous equation, commuting derivatives we get
\begin{equation}
\label{eq-der}
\begin{split}
\frac{\partial}{\partial t} \nabla_r h_i^j &=  
\mathcal{L}_{\e}(\nabla_r h_i^j) + \frac{\partial^2\kappa_{\e}}{\partial h_p^q
\partial h_n^s}\nabla_r h_n^s\nabla_p\nabla_q h_i^j + 
\frac{\partial \kappa_{\e}}{\partial h_{pq}}R_{rpqm}\nabla_m h_i^j \\
&+ \frac{\partial^3 \kappa_\e}{\partial h_p^q\partial h_m^l\partial h_n^s}\nabla_rh_n^s\nabla^ih_p^q\nabla_j h_m^l  +
\frac{\partial^2\kappa_{\e}}{\partial h_pq\partial h_m^l}\nabla_r\nabla^ih_p^q \nabla_j h_m^l + \\
&+ \frac{\partial^2\kappa_{\e}}{\partial h_pq\partial h_m^l}\nabla^ih_p^q \nabla_r\nabla_j h_m^l +
\frac{\partial^2\kappa_{\e}}{\partial h_m^l\partial h_n^s}\nabla_r h_n^sh_p^lh_m^ph_j^i + \\
&+ \frac{\partial\kappa_{\e}}{\partial h_m^l}\nabla_r h_p^l h_m^p h_i^j 
+ \frac{\partial\kappa_{\e}}{\partial h_m^l}h_p^l\nabla_rh_m^p h_i^j +
 \frac{\partial\kappa_{\e}}{\partial h_m^l}h_p^l h_m^p \nabla_rh_i^j.
\end{split}
\end{equation}
Let $w = \sum_{i,j}|\nabla h_i^j|^2$. Since $|\nabla h_i^j|^2 = g^{pq}\nabla_p h_i^j\nabla_q h_i^j$ and 
$\frac{\partial g_{ij}}{\partial t} = 2\kappa_{\e}h_{ij}$, we get
\begin{equation*}
\begin{split}
\label{eq-nabla}
\frac{\partial w}{\partial t}  &= -4g^{pa}g^{qb}\kappa_{\e}h_{ab}\nabla_p h_i^j\nabla_q h_i^j
+ \mathcal{L}_{\e}(w)  - 2\dot{\kappa}_{\e}(\nabla^2h_i^j, \nabla^2 h_i^j) +   \\
&+ g^{pq}\frac{\partial\kappa_{\e}}{\partial h_a^b}R_{pabs}\nabla_sh_i^j\nabla_q h_i^j + 
\frac{\partial^2\kappa_{\e}}{\partial h_p^q\partial h_m^l} g^{ra}\nabla_r h_m^l 
\nabla_p\nabla_q h_i^j \nabla_a h_i^j +   \\
&+ \frac{\partial^3}{\partial h_p^q\partial h_m^l \partial h_n^s} g^{ra}\nabla_r h_n^s 
\nabla^i h_p^q\nabla_j h_m^l \nabla_a h_i^j
+ \frac{\partial^2\kappa_{\e}}{\partial h_p^q\partial h_m^l} g^{ra} 
\nabla_r\nabla^i h_p^q\nabla_j h_m^l\nabla_a h_i^j  +   \\
&+ \frac{\partial^2\kappa_{\e}}{\partial h_p^q\partial h_m^l} g^{ra} 
\nabla^i h_p^q\nabla_r\nabla_j h_m^l\nabla_a h_i^j 
+ \frac{\partial^2\kappa_{\e}}{\partial h_m^l\partial h_n^s}g^{ra}\nabla_rh_n^s h_p^lh_m^p h_i^j +   \\
&+ \frac{\partial\kappa_{\e}}{\partial h_m^l}g^{ra}\nabla_r h_p^l h_m^p h_i^j \nabla_a h_i^j +
\frac{\partial\kappa_{\e}}{\partial h_m^l}g^{ra} h_p^l \nabla_r h_m^p h_i^j \nabla_a h_i^j  +   \\
&+  \frac{\partial\kappa_{\e}}{\partial h_m^l}g^{ra} h_p^l h_m^p \nabla_rh_i^j \nabla_a h_i^j. 
\end{split}
\end{equation*}
Whenever we see $i$ and $j$ in the previous equation we assume that we are summing over all indices $i$ and $j$. Also,  
$$\dot{\kappa}_{\e}(\nabla^2 h_i^j, \nabla^2 h_i^j) = \frac{\partial\kappa_{\e}}{\partial h_p^q}g^{pq}g^{cd}
\nabla_q\nabla_c h_i^j\nabla_p\nabla_d h_i^j.$$ 
Notice that since $|A| \le C$ for all $t\in [0, \min (T_{\e},T_0)\, )$, we have 
$$|\frac{\partial^2\kappa_{\e}}{\partial h_p^q\partial h_m^l}| \le \frac{C_1}{H^3} 
\quad \mbox{and} \quad |\frac{\partial^3\kappa_{\e}}{\partial h_p^q\partial h_m^l\partial h_n^s}| \le \frac{C_1}{H^4}$$
for a uniform constant $C$.

We next compute the evolution equation for ${1}/{H}$ from the evolution of
$H$, namely  equation \eqref{eqn-H}. By direct computation we get
that
$$
\frac{\partial}{\partial t}(\frac{1}{H}) = 
\mathcal{L}_{\e}(\frac{1}{H}) - \frac{2}{H^3}\frac{\partial\kappa_{\e}}{\partial h_p^q}
\nabla_p H\nabla_q H - \frac{1}{H^2}\frac{\partial^2\kappa_{\e}}
{\partial h_p^q\partial h_m^l}\nabla^i h_p^q\nabla_i h_m^l
- \frac{1}{H}\frac{\partial\kappa_{\e}}{\partial h_m^l} h_p^l h_m^p.$$  
Taking away the second negative term on the right hand side we easily conclude the
differential inequality
$$\frac{\partial}{\partial t}(\frac{1}{H})\le \mathcal{L}_{\e}(\frac{1}{H}) + \frac{C\,w}{H^5} + \frac{C}{H}.$$
Combining   the evolution equations of $w$ and $1/H$  we will  now compute  the evolution equation for 
$$\mathcal{V} := w + \frac{1}{H}.$$ We look at the point $(P,t)$ at which $\mathcal{V}$ 
achieves its maximum at time $t$ and choose coordinates around $P$ so that both,
the second fundamental form and the metric matrix are diagonal at $P$.  Using 
the exact form of coefficients $a^{ik}_\e = \frac{\partial\kappa_{\e}}{\partial h_p^q}$
computed in (\ref{eq-coeff}) we get
\begin{equation}
\begin{split}
\label{eq-good}
-2\sum_{i,j}\dot{\kappa}_{\e}(\nabla^2 h_i^j, \nabla^2 h_i^j) &=
-2\sum_{i,j} \frac{\partial\kappa_{\e}}{\partial h_p^q} 
g^{cd}\nabla_q\nabla_c h_i^j \nabla_p\nabla_d h_i^j\\
&= -\frac{2}{H^2}\sum_{i,j}[\, \lambda_2^2 (\nabla_1\nabla_1h_i^j)^2 +
\lambda_1^2(\nabla_2\nabla_2 h_i^j)^2 \\
&\,\,\, \quad+ (\lambda_1^2 +
\lambda_2^2)\, (\nabla_1\nabla_2 h_i^j)^2 \, ].
\end{split}
\end{equation}
Our goal is to absorb all the remaining terms that contain the second order derivatives,
appearing  in the evolution equation for $\mathcal{V}$, in the good term (\ref{eq-good}).
By looking at the evolution equation of $w$,  we see that those second order terms are
$$\mathcal{O} = \frac{\partial^2\kappa_{\e}}{\partial h_p^q\partial h_m^l}
g^{ra}\nabla_r h_m^l\nabla_p\nabla_q h_i^j\nabla_a h_i^j,$$
$$\mathcal{P} = \frac{\partial^2\kappa_{\e}}{\partial h_p^q\partial h_m^l}
g^{ra}\nabla_r\nabla^i h_p^q\nabla_j h_m^l\nabla_a h_i^j$$
and 
$$\mathcal{R} = \frac{\partial^2\kappa_{\e}}{\partial h_p^q\partial h_m^l}
g^{ra}\nabla^ih_p^q\nabla_r\nabla_jh_m^l\nabla_a h_i^j$$
where we understand summing over all indices.
Denote by $\xi_m^l := \nabla_r h_m^l$ and by $\eta_p^q := \nabla_p\nabla_q h_i^j$.
If we specify the coordinates around the maximum point $P$ in which $W$ and $g$
are diagonal, by (\ref{eq-ddot})
\begin{equation}
\begin{split}
\label{equation-O}
\mathcal{O} &= g^{ra}\nabla_r h_i^j \,  \left (\sum_{p,q}\frac{\partial^2\kappa}{\partial\lambda_p\lambda_q}
\xi_p^p\eta_q^q + \sum_{p\neq q} \frac{\frac{\partial\kappa}{\partial\lambda_p}
- \frac{\partial\kappa}{\partial\lambda_q}}{\lambda_p - \lambda_q}\xi_p^q\eta_p^q \right ) \\
&= \nabla_r h_i^j \left (-\frac{2\lambda_2^2}{H^3}\nabla_r h_1^1\nabla_1\nabla_1 h_i^j
- \frac{2\lambda_1^2}{H^3}\nabla_r h_2^2\nabla_2\nabla_2 h_i^j +
\frac{2\lambda_1\lambda_2}{H^3}\nabla_r h_1^1\nabla_2\nabla_2 h_i^j + 
\right .\\
&+ \left . \frac{2\lambda_1\lambda_2}{H^3}\nabla_r h_2^2\nabla_1\nabla_1 h_i^j 
- \frac{1}{H}(\nabla_1\nabla_2 h_i^j\nabla_r h_1^2 + \nabla_2\nabla_1 h_i^j
\nabla_r h_2^1 \right ). 
\end{split}
\end{equation}
Since $|A| \le C$ and $$\frac{1}{H} = \frac{\lambda_1 + \lambda_2}{H^2} \le 
\frac{|\lambda_1| + |\lambda_2|}{H^2}$$
by the  Cauchy-Schwartz inequality,
we can estimate $\mathcal{O}$ term by term, namely 
\begin{eqnarray*}
\left | 2\nabla_r h_i^j\frac{\lambda_2^2}{H^3}\nabla_r h_1^1\nabla_1\nabla_1 h_i^j \right | &=&
\left  |2\nabla_r h_i^j\frac{\lambda_2}{H^2}\nabla_r h_1^1 \right |\, \left |\frac{\lambda_2}{H}\, 
\nabla_1\nabla_1 h_i^j \right | \\
&\le& C\frac{w^2}{H^4} + \beta_1\frac{\lambda_2^2}{H^2}|\nabla_1\nabla_1 h_i^j|^2 \\
&\le&  C\frac{w^2}{H^4} + \beta_1\sum_{ij}\dot{\kappa_\e}(\nabla^2h_i^j\nabla^2 h_i^j)
\end{eqnarray*}
and
\begin{eqnarray*}
\left |2\frac{\lambda_1\lambda_2}{H^3}\nabla_r h_i^j\nabla_r h_1^1\nabla_1\nabla_1 h_i^j \right | &=&
\left |2\frac{\lambda_1}{H^2}\, \nabla_r h_1^1\nabla_r h_i^j \right | \, \left |\frac{\lambda_2}{H}\nabla_1\nabla_i h_i^j \right | \\
&\le& C\frac{w^2}{H^4}  + \beta_1\frac{\lambda_2^2}{H^2}|\nabla_1\nabla_1 h_i^j| \\
&\le& C\frac{w^2}{H^4} + \beta_1\sum_{ij}\dot{\kappa}_\e(\nabla^2 h_i^j\nabla^2 h_i^j)
\end{eqnarray*}
and
\begin{eqnarray*}
\left | \frac{1}{H}\nabla_1\nabla_2 h_i^j\nabla_rh_1^2\nabla_rh_i^j \right | &\le&
\left |\frac{1}{H^2}\nabla_r h_1^2\nabla_r h_i^j \right |\, \left  |\frac{\lambda_1 + \lambda_2}{H}\nabla_1\nabla_2 h_i^j \right |\\
&\le& C\frac{w^2}{H^4} + \beta_1(\lambda_1^2 + \lambda_2^2)|\nabla_1\nabla_2 h_i^j|^2  \\
&\le& C\frac{w^2}{H^4} + \beta_1\sum_{ij}\dot{\kappa}_\e(\nabla^2 h_i^j,\nabla^2 h_i^j)
\end{eqnarray*}
where $\beta_1 > 0$ is a uniform small number. We can estimate other terms in $\mathcal{O}$ 
the same way and combining all those estimates yield
\begin{equation*}
|\mathcal{O}| \le C\frac{w^2}{H^4} + \beta \sum_{i,j}
\dot{\kappa}_{\e}(\nabla^2 h_i^j, \nabla^2 h_i^j)
\end{equation*}
where $\beta > 0$ is a small fixed number. 

In order to estimate $\mathcal{P}$, we would like to be able somehow to switch
the pair of indices $\{i,r\}$ with $\{p,q\}$ so that we reduce estimating 
$\mathcal{P}$ to the previous case of $\mathcal{O}$. We will use Gauss-Codazzi 
equations in the form
$$\nabla_l h_{ij} = \nabla_i h_{lj}.$$
In our special coordinates at the point we have 
\begin{equation}
\begin{split}
\label{eq-reduction}
\nabla_r\nabla^i h_p^q &= \nabla_r\nabla^i(h_{ps}g^{qs}) =
\nabla_r(g^{ij}\nabla_j(h_{ps}g^{qs})) \\
&= \nabla_r(g^{ij}g^{qs})
\cdot \nabla_j h_{ps} + g^{ij}g^{qs}\nabla_r\nabla_j h_{ps} +
h_{ps}\nabla_r(g^{ij}\nabla_j g^{qs}) \\
&= \nabla_p\nabla_q h_{rj} + \nabla_r(g^{ij}g^{qs})
\cdot \nabla_j h_{ps} + h_{pp}\nabla_r(g^{ij}\nabla_j g^{pq}) 
\end{split}
\end{equation}

\noindent We have the following:

\noindent{\em Claim.}
There is a uniform constant $\tilde{C}$ depending on $C$, so that
\begin{equation} \label{claim-bound-g}
|g(\cdot,t)|_{C^2} \le \tilde{C}
\end{equation}
as long as $|A| \le C$.

\noindent 
To prove \eqref{claim-bound-g} we observe that in geodesic
coordinates $\{x_i\}$ around a point $p$, which corresponds to the
origin in geodesic coordinates, we have
\begin{equation}
\label{equation-g}
g_{ij}(x) = \delta_{ij} + \frac{1}{3}R_{ipqj}x^px^q + O(|x|^3)
\end{equation}
and that an easy computation shows that
$$\nabla_p\nabla_q g_{ij} (0) = -\frac{1}{3}R_{ipqj}.$$
By the Gauss equations, we have $R_{ipqj} = h_{iq}h_{pj} - h_{ij}h_{pq}$, which yields to 
$|\nabla_p\nabla_q g_{ij}| \le \tilde{C}$ as long as  $|A| \le C$. This together
with (\ref{equation-g}) proves the Claim.

\medskip

Combining  \eqref{eq-reduction}- \eqref{claim-bound-g}, we obtain as in the
estimate of $\mathcal{O}$, the bound 
$$|\mathcal{P}| \le \frac{Cw^2}{H^4} + \beta\sum_{i,j}\dot{\kappa}_{\e}
(\nabla^2 h_i^j, \nabla^2 h_i^j).$$
Similarly, we get the estimate for $\mathcal{R}$. We conclude that 
\begin{equation}
\label{eqn-oooo}
|\mathcal{O}| + |\mathcal{P}| + |\mathcal{R}| \le 
\frac{Cw^2}{H^4} + 3\beta\sum_{i,j}\dot{\kappa}_{\e}
(\nabla^2 h_i^j, \nabla^2 h_i^j).
\end{equation}
Choosing  $\beta > 0$ so that $3\beta < 2$ in \eqref{eqn-oooo} and analyzing  the right hand side
of the evolution of $w$ term by term,  we obtain the 
following estimate at  the maximum point $P$  of $\mathcal{V}$ at time $t$ 
$$\frac{d \, \mathcal{V}_{\max}}{dt} \le Cw  + \frac{Cw^2}{H^4}
+ \frac{C\sqrt{w}}{H^3} + \frac{Cw}{H^5} + \frac{C}{H}.$$
Young's inequlity, implies the estimates
$$\frac{w^2}{H^4} \le w^6 + \frac{1}{H^6} \le \mathcal{V}^6$$
and
$$\frac{w}{H^5} \le w^6 + \frac{1}{H^6} \le \mathcal{V}^6$$
and
$$\frac{\sqrt{w}}{H^3} \le w + \frac{1}{H^6} \le \mathcal{V} + \mathcal{V}^6.$$

\smallskip

\noindent Hence, denoting  by $f(t) = \mathcal{V}_{\max}(t)$ we obtain
\begin{equation}
\label{eq-good-f}
\frac{d f}{dt} \le C\, (f + f^6)
\end{equation}
which implies the existence of uniform constants $\bar C$ and $T_0$,
depending only on $C$ and $f(0)$, so that
$$\sup_{\Sigma_t^{\e}}\, (\frac{1}{H} + \sum_{i,j}|\nabla h_i^j|^2) \le \bar{C},
\qquad \mbox{for all} \,\,  t \in [0,\min  (T_\e, T_0)\, ).$$
This finishes the proof of the lemma. 
\end{proof}



Having all the curvature estimates  (that are proved above),
we can justify the short time existence of the $C^{2,1}$-solution
to the (\ref{equation-HMCF}).

\begin{proof}[Proof of Theorem \ref{thm-STE}]
For every $\e > 0$,  let $T_{\e}$ be the maximal time so that 
$$|A|_{C^1(\Sigma_t^{\e})} \le   C, \qquad \mbox{and} \qquad H \ge \delta >0$$
where $C, \delta$ are constants taken from  Proposition \ref{prop-nabla-curv}. 
Take now $\e_i\to 0$.
We have that $|A|_{C^1(\Sigma_t^{\e_i})} \le C$,
which implies $|F_{\e_i}|_{C^{2,1}} \le C$, for all   $t\in [0,T_0]$.
By the Arzela-Ascoli theorem there is a subsequence so that 
$F_{\e_i}(\cdot,t)\stackrel{C^{2,1}}{\to} F(\cdot,t)$,
where $F(\cdot,t)$ is a $C^{2,1}$ solution to (\ref{equation-HMCF}).
Since we have a comparison principle for $C^{2,1}$ solutions to 
(\ref{hmcf0}) as discussed above, the uniqueness of a $C^{2,1}$ solution 
immediately follows.
\end{proof}


\section{Long time existence for the $\e$-flow}\label{section-ltee}
In this section we will study the long time existence for the
$\e$-regularized flow \eqref{equation-reg} assuming that $\Sigma_0$ is
an arbitrary smooth surface with mean curvature $H > 0$, Euler
characteristic $\chi(\Sigma_0) > 0$ and it is star-shaped with respect
to the origin. Throughout the section we fix $\e >0$ sufficiently
small, we denote by $\Sigma_t^\e$ the surface evolving by
\eqref{equation-reg} and, to simplify the notation, we drop the index
$\e$ from $F,\nu,H,G,\kappa,A,g_{ij}, h_{ij}$ etc.  The $\e$-flow has
one obvious advantage over (\ref{hmcf0}), it is 
not degenerate and therefore it has smoothing properties. Indeed, it
follows from the Krylov and Schauder estimates that a $C^{1,1}$
solution of \eqref{equation-reg} is $C^\infty$ smooth.

Assume that $\Sigma_t^\e$ is a solution of \eqref{equation-reg} on
$[0,T_\e)$ and let us consider the evolution equation for the area
form $d\mu_t$, namely
$$\frac{\partial }{\partial t} \, d\mu_t  = -2 \, ( \frac{G}{H} + \e\, H)\,  \frac{H}{2} \, d\mu_t 
= -(G + \e\, H^2 ) \, d\mu_t.$$ Integrating it over the surface
$\Sigma_t^\e$ we obtain the following ODE for the total area
$\mu_t(\Sigma_t^\e)$ of the surface $\Sigma_t^\e$
$$\frac{d}{d t}\,  \mu_t(\Sigma_t^\e)  = - \int_{\Sigma_t^\e} (G + \e\, H^2 ) \, d\mu_t.$$
By the Gauss-Bonnet formula we have
$$ \int_{\Sigma_t^\e}  G  \, d\mu_t = 2\pi\, \chi(\Sigma_t).$$ 
Since $\Sigma_0$ is a surface with positive Euler characteristic, then
by the uniformization theorem $\chi(\Sigma_t) = 2$ and therefore we
conclude the equation
\begin{equation}
\label{eqn-vol}
\frac{d}{d t} \,  \mu_t(\Sigma_t^\e) = -  4\, \pi -  \e\, \int_{\Sigma_t^\e}  H^2  \, d\mu_t.
\end{equation}
Denote by $T_\e$ the maximum time of existence of \eqref{equation-reg}. Integrating
\eqref{eqn-vol} in time from $0$ to $T_\e$,  solving with respect of $T_\e$ and
using that $\mu_t(\Sigma_t^\e) \geq 0$, gives
$$T_\e \leq \frac 1{4\pi}  \mu_0(\Sigma_0) - \frac \e{4\pi} \, \int_0^{T_\e} \int_{\Sigma_t^\e}  H^2  \, d\mu_t.$$
This, in particular shows that
\begin{equation}\label{eqn-Te}
T_\e \leq \frac 1{4\pi} \mu_0(\Sigma_0) 
\end{equation}
where $ \mu_0(\Sigma_0) $ is the area of the initial surface $\Sigma_0$. 

Our goal is to prove the following result, concerning the long time existence of
the flow \eqref{equation-reg}. We will also establish curvature bounds and curvature pinching estimates which are independent of $\e$. 

\begin{thm}\label{thm-e}
Let $\Sigma_0$ be a compact  star-shaped  hyper-surface in $\R^3$ which is of class $C^{1,1}$  and
has strictly positive mean curvature  $H > 0$. Then, there exists a maximal time 
of existence $T_\e$  of a smooth \eqref{equation-reg} flow $\Sigma_{\e}^t$ such that
either:
\begin{enumerate}[(i)]
\item $H(P_t,t) \to 0$, as $t \to T_\e$  at some points $P_t  \in \Sigma_t^\e$, or
\item $\Sigma_{\e}^t$ shrinks  to a point as
$t \to T_{\e}$ and  $T_\e$ is given explicitly by
\begin{equation}
\label{eq-ext-time}
T_\e =  \frac 1{4\pi}  \mu_0(\Sigma_0)- \frac {\e}{4\pi} \, \int_0^{T_\e} \int_{\Sigma_t^\e}  H^2  \, \mu_t
\end{equation}
where $\mu_0(\Sigma_0)$ is the total area of $\Sigma_0$. Moreover,   $\int_{\Sigma_t^{\e}} H^2\, d\mu_t$ is uniformly bounded for all $t\in [0,T_{\e})$, independently of $\e$. 
\end{enumerate}
\end{thm}

Assume that (i) does not happen in Theorem \ref{thm-e}. Then, we have  
\begin{equation}
\label{eq-assump}
\min_{\Sigma_{\e}^t} H(\cdot,t)  \geq \delta >0, \,\,\, \mbox{for all} \,\,\, 
t\in [0,T_{\e})
\end{equation}
where $T_{\e}$ is the maximal existence time of a smooth flow
$\Sigma_{\e}^t$.

\begin{prop} 
\label{thm-max-time} Assuming that (i) doesn't happen in Theorem \ref{thm-e}, then the
maximal time of existence $T$ of the flow \eqref{equation-reg} satisfies 
$T  \leq \mu_0(\Sigma_0)/4\pi$ and
$$\limsup_{t\to T}|A| = \infty.$$
\end{prop}

\begin{proof}
The bound $T \leq M_0/4\pi$ is proven above. 
Assume that $\max_{\Sigma_{\e}^t}|A| \le C$ for all $t\in [0,T)$. Then we 
want to show that the surfaces $\Sigma^{\e}_t$ converge, as $t \to T$, to a smooth limiting surface
$\Sigma^\e_T$.  Similarly as in \cite{Di},
using the curvature bounds we have,  for all $0 < t_t < t_2 <T$,  the bounds
$$|F(p,t_1) - F(p,t_2)| \le C\, |t_2 - t_1| \qquad \mbox{and} \qquad  |\frac{\partial}{\partial t}g_{ij}|^2 \le C$$
for a uniform in $t$  constant $C$, which imply that  $F(\cdot,t)$ converges,
as $t\to T$ to some continuous surface $\tilde \Sigma^{T}_\e$. We get uniform $C^2$-bounds on $F$
out of the bound on $|A|$. Since our equation is uniformly parabolic and
 the operator $\kappa$ is concave, by Krylov and Schauder estimates
we obtain all higher derivative bounds.

We have just shown that the surface $\Sigma_T^\e$ is $C^\infty$ smooth. 
Also from our assumption 
$$H(\cdot,T) \geq \delta >0, \qquad \mbox{on} \,\, \Sigma_T^\e.$$
By Proposition \ref{prop-ste-e}  there exists $\tau_\e >0$ for which a smooth  flow can be continued
on $[T,T+\tau_\e)$, which contradicts our assumption that $T$ is maximal. Hence,
$\limsup_{t\to T_{\e}}|A| = \infty$ and the result follows. 

\end{proof}


\subsection{Monotonicity formula}
We will now show the  monotonicity property of the quantity 
$$Q_\e = \langle F_{\epsilon},\nu \rangle + 2 t \, \kappa_\epsilon$$
along the flow (\ref{equation-reg}).  This will play an  essential role in
establishing the long time existence.  Similar
quantity was considered by Smoczyk in \cite{Sm}.

\begin{lem}
\label{lemma-monotone}
Assuming that  $q_\e(0):=\min_{\Sigma_0} \langle F_{\epsilon},\nu\rangle  \geq 0 $, 
the quantity $$q_\e(t):=\min_{\Sigma_t^\e} (\langle F_{\epsilon},\nu\rangle + 2 t \, \kappa_\epsilon) $$
is increasing in time
for as long as the solution $\Sigma_t^\e$ exists. Hence, $$q_\e(t):=\min_{\Sigma_t^\e} (\langle F_{\epsilon},\nu\rangle + 2 t \, \kappa_\epsilon) \geq q_\e(0) \geq 0.$$
\end{lem}

\begin{proof} We will compute the evolution of $Q_\e$ and apply the maximum principle.
We begin by computing the evolution of  $\langle F_{\epsilon},\nu\rangle$. We have:
$$\mathcal{L}_{\epsilon}(\langle F_{\epsilon},\nu\rangle) = 
a_{\epsilon}^{ik}\nabla_i(\nabla_k \langle F_{\epsilon},\nu\rangle)
= a_{\epsilon}^{ik}\nabla_i(\langle e_k, \nu\rangle + \langle
F_{\epsilon},h_{kj}e_j\rangle),$$ since
$$\nabla_i\nu = h_{ij}e_j, \qquad \nabla_ie_j = -h_{ij}\nu, \qquad \nabla_i F_{\epsilon} = e_i.$$
Using Gauss-Codazzi equation $\nabla_i h_{kj} = \nabla_j h_{ik}$ and
since $\nabla_i(\langle e_k,\nu\rangle) = 0$ we get
\begin{eqnarray*}
\mathcal{L}_{\epsilon}(\langle F_{\epsilon},\nu\rangle) &=& a_{\epsilon}^{ik}
[\langle \nabla_iF_{\epsilon},h_{kj}e_j\rangle
+ \langle F, h_{kj}\nabla_i e_j \rangle + \langle F, e_j\cdot\nabla_i h_{jk}\rangle] \\
&=& a_{\epsilon}^{ik}[h_{ik} - h_{ij}h_{jk}\langle F_{\epsilon},\nu\rangle + 
\langle F, e_j\cdot \nabla_j h_{ik}\rangle] \\
&=& \kappa_\epsilon - a^{ik}h_{ij}h_{jk}\langle F_{\epsilon},\nu\rangle + 
\langle F_{\epsilon}, \nabla \kappa_\epsilon\rangle. 
\end{eqnarray*}
On the other hand, since $\frac{\partial\nu}{\partial t} = \nabla
\kappa_\epsilon$, it follows 
$$\frac{\partial}{\partial t}\langle F_{\epsilon},\nu\rangle = 
-\kappa_\epsilon + \langle F_{\epsilon}, \nabla \kappa_\epsilon\rangle$$
which yields
$$\frac{\partial}{\partial t}\langle F_{\epsilon},\nu\rangle - \mathcal{L}_\epsilon(\langle F_{\epsilon},\nu\rangle) =
-2\kappa_\epsilon + a^{ik}h_{ij}h_{jk}\langle F_{\epsilon},\nu\rangle.$$
We also have
$$\frac{\partial \kappa_\epsilon}{\partial t} = \mathcal{L}_{\epsilon}(\kappa_\epsilon) 
+ a_{\epsilon}^{ik}h_{ij}h_{jk}\kappa_\epsilon.$$
Hence  $Q_{\epsilon} = \langle F_{\epsilon},\nu\rangle + 2t\,  \kappa_\epsilon$ satisfies
\begin{equation}
\label{equation-Q}
\frac{\partial Q_{\epsilon}}{\partial t} = \mathcal{L}_{\epsilon} Q_\epsilon+ a_{\epsilon}^{ik}\nabla_i\nabla_k Q_{\epsilon} + 
a_{\epsilon}^{ik}h_{ij}h_{jk} Q_{\epsilon}.
\end{equation}
Notice that  the right hand side of (\ref{equation-Q}) is a strictly elliptic
operator and $$a_{\epsilon}^{ik}h_{ij}h_{jk} \geq 0.$$ We conclude by   the maximum principle that 
$$q_\e'(t) = \frac{d}{dt}(\, \langle F_{\epsilon},\nu\rangle + 2t\,  \kappa_\epsilon\rangle \, )_{\min} \ge 0$$
assuming that $q_\e(0) \geq 0$. This implies that $q_\e(t) \geq q_\e(0)$ finishing the proof of
the lemma.  
\end{proof}

Notice that if instead of $Q_{\epsilon}$ we take the quantity 
$$Q_{\eta, \e} =  \langle F_{\epsilon},\nu \rangle + 2(t + \eta) \,  \kappa_\epsilon$$
for any constant $\eta \in \R$, the same computation as above yields to 
that $Q_{\eta,\e}$ satisfies
$$\frac{\partial Q_{\eta,\e}}{\partial t} =\mathcal{L}_{\epsilon} Q_{\eta,\epsilon}  + a
_{\epsilon}^{ik}\nabla_i\nabla_k Q_{\eta,\e} + a_{\epsilon}^{ik}h_{ij}h_{jk}Q_{\eta,\e}.$$
Assume that at time $t=0$, we have 
$$q_{\eta, \e}(0) =  \min_{\Sigma_0} (\langle F,\nu \rangle + 2 \, \eta  \,  \kappa_\e)  \geq 0$$
for some $\eta \in \R$ (notice that  $F_\e=F$ at $t=0$). 
Then,   the maximum principle to the
above equation, gives:

\begin{prop}\label{rem-useful}
For any $\eta \in \R$, such that $q_{\eta, \e} (0):=
 \min_{\Sigma_0} \,  ( \langle F ,\nu \rangle + 2\,  \eta \,  \kappa_\e ) \geq 0 $
  the quantity 
$$q_{\eta, \e}(t) := \min_{\Sigma^\e_t} \,  (\langle F_{\epsilon},\nu \rangle + 2(t + \eta) \,  \kappa_\epsilon\, )$$
is increasing in time. Hence
\begin{equation}
\label{equation-cases}
{q}_{\epsilon,\eta}(t) := \min_{\Sigma^\e_t} \,  (\langle F_{\epsilon},\nu \rangle + 2(t + \eta) \,  \kappa_\epsilon\, ) 
\ge {q}_{\epsilon,\eta}(0) \geq 0.
\end{equation}

\end{prop}

\medskip
Since the initial surface $\Sigma_0$ is star-shaped, we may  choose $\eta > 0$ 
so that we have $q_{\eta,\e}(0)  > 0$. This is possible by continuity, since  $\langle F,\nu\rangle > 0$.
By Proposition  \ref{rem-useful} we have 
\begin{equation*}
{Q}_{\eta,\e}(\cdot,t) \ge q_{\eta,\e}  (t) >0
\end{equation*}
which  implies the lower bound 
\begin{equation}
\label{equation-better1}
\kappa_\e = \frac{G}{H} + \epsilon H \ge - \frac{\langle F_{\epsilon},\nu\rangle}{2 (t+\eta)}.
\end{equation}
We will show  next that \eqref{equation-better1} implies  a uniform lower bound on $\kappa_{\e}$, independently of $\e$. To this end,
we need to bound $\langle F_{\e},\nu\rangle$, independently of $\e$. This bound follows from
the  comparison principle for curvature  flows  with the property that the speed is 
an increasing function of the principal curvatures. More precisely, we have the following lemma.

\begin{lem}
\label{lem-bound-F}
There exists  a constant $C$,   independent of $\epsilon$,  so that 
\begin{equation}
\label{equation-lower-speed}
\kappa_\epsilon:=\frac{G}{H} + \epsilon\,  H \ge -C, \qquad \forall t \in [0,T_\epsilon).
\end{equation}
\end{lem}

\begin{proof}
We claim that there is a uniform constant $C$ so that $|F_\e| \le
C$, for all $t \in [0,T_\e)$.  To see that, let $\psi_0: S^2 \to
\mathbb{R}^3$ denote the parametrization of a sphere that encloses
the initial hypersurface $\Sigma_0$.  By the result of B.  Andrews
in \cite{An1}, the solution $\psi_\e(\cdot,t)$ of
(\ref{equation-reg}) with initial condition $\psi_0$ shrinks to a
point in some finite time $\tilde{T}_\e$. Moreover,
\begin{equation}\label{eqn-tildeT}
\tilde T_\e \leq \tilde T < \infty
\end{equation}
for a uniform constant $\tilde T$. 

The standard comparison principle shows that the images of $F_\e$ and
$\psi_\e$ stay disjoint for all the time of their existence. To see
this, we consider the evolution of $$d(p,q,t) := |F_\e(p,t) -
\psi_\e(q,t)|, \qquad (p,q) \in \Sigma_t^\e \times S^2.$$ Assume that
the minimum of $d$ at time $t$ occurs at $(p_0,q_0)$. If $W$ denotes
the Weingarten map, at that minimum point $W(p_0) \ge W(q_0)$, so by
the monotonicity of our speed $\kappa_\e$, $\kappa_\e(W(p_0)) \ge
\kappa_\e(W(q_0))$. The maximum principle tells us that $d_{min}(t)$
is non-decreasing and therefore the distance between the images of
$F_{\e}$ and $\psi_\e$ is non-decreasing. Hence, they stay disjoint in
time.  As a consequence of that, our hypersurfaces $F_\e(\cdot,t)$
stay enclosed by the sphere $\psi_0$ for all times of their existence
(since $\psi_\e(\cdot, t)$ are enclosed by $\psi_0$) and therefore
$|F_\e|(\cdot,t) \le C$ for a uniform constant $C$.

The above bound implies that $\langle F_\e, \nu \rangle \leq C$, for a
uniform in $\e$ constant $C$ and all $t \in [0,T_\e)$.  This together
with (\ref{equation-better1}) yield to (\ref{equation-lower-speed}).
\end{proof}

\subsection{Curvature pinching estimates for the $\e$-flow}
\medskip
Define  $$\mathcal{F_\e} := \langle F_\e,\nu\rangle + 2t\, \kappa_\e.$$
Notice that division by $\mathcal{F_\e}$ makes sense since by Lemma
\ref{lemma-monotone}, $(\mathcal{F_\e})_{\min}$ is increasing in time and
$(\mathcal{F_\e})_{\min}(0) \ge \delta > 0$ due to star-shapedness. 
As we showed above, $\sup_{\Sigma_t}|F_\e| \le C$
for a uniform constant $C$. Rewrite the evolution equation for $H$ 
from Lemma \ref{lemma-evolution} in the form
$$\frac{\partial}{\partial t}H = \mathcal{L}_\e(H) + \ddot{\kappa}_\e(\nabla W, \nabla W)
+ \dot{\kappa}_\e(W^2)H$$
where $W$ is the Weingarten map and 
$$\ddot{\kappa}_\e(\nabla W, \nabla W) =
\frac{\partial^2 \kappa_\e}{\partial h_q^p\partial h_m^l}
\nabla^i h_q^p\nabla_j h_m^l \quad \mbox{and} \quad \dot{\kappa}_\e(W^2) =
\frac{\partial\kappa_\e}{\partial h_m^l}h_p^lh_m^p.$$
Then, by  direct computation  we have  
\begin{equation}
\label{eq-quotient}
\frac{\partial}{\partial t} \left ( \frac{H}{\mathcal{F}_\e} \right ) =
\mathcal{L}_\e\left (\frac{H}{\mathcal{F}_\e} \right ) + \frac{2}{\mathcal{F}_\e}\, \dot{\kappa}_\e
\left (\nabla\mathcal{F}, \nabla\frac{H}{\mathcal{F}_\e} \right  )+
\frac{1}{2\mathcal{F}_\e}\, \tr_g\ddot{\kappa_\e}(\nabla W, \nabla W).
\end{equation} 
By (\ref{eq-definite-H}), the last term in this equation is negative.
Hence, by the maximum principle, the supremum of
${H}/{\mathcal{F_\e}}$ is decreasing. In particular, we have: 

\begin{lem} Assume that $\Sigma_t^\e$  is a solution of  \eqref{equation-reg}  on
$[0,T_\e)$ with $\Sigma_0$ as in Theorem \ref{thm-STE}. Then, 
\begin{equation}
\label{eq-better-H}
\sup_{\Sigma_t^\e\times[0,\tau)}\frac{H}{\langle F_\e, \nu\rangle + 2t\, \kappa_\e} \le C,
\qquad \mbox{on} \,\,\, [0,T_\e)
\end{equation}
for a uniform constant $C$ that depends only on $\Sigma_0$. 
\end{lem}

Denote by $\lambda_1, \lambda_2$ the two principal curvatures of the surface $\Sigma_t^\e$ at some time $t$ and point $P$.

\begin{lem}\label{lem-111}
If there is some time $t_0$ so that $\liminf_{t\to t_0} H(\cdot,t) = 0$, then 
\begin{equation}
\label{eq-eigen-zero}
\liminf_{t\to t_0} \, (\lambda_1^2 + \lambda_2^2)  = 0.
\end{equation}
\end{lem}

\begin{proof}
Assume $\liminf_{t\to t_0} H(\cdot,t) = 0$. We distinguish the following two cases:
\begin{enumerate}
\item[(i)]
$\lambda_1 > 0$ and $\lambda_2 \ge 0$. In this case (\ref{eq-eigen-zero})
immediately follows.
\item[(ii)]
$\lambda_1 > 0$ and $\lambda_2 < 0$. By Lemma \ref{lemma-monotone},
$$\kappa_\e:=\frac{G}{H} + \e\, H \ge -C$$ uniformly in time, which implies
$$\lambda_1\, |\lambda_2| \le C\,  H + \e \, H^2.$$
Since $\liminf_{t\to t_0} H = 0$, at least for one of the two principal curvatures  
must tend to zero, i.e. 
\begin{equation}
\label{eq-both-zero}
\liminf_{t\to t_0} |\lambda_i| =0.
\end{equation}
Since $\liminf_{t\to t_0} H = 0$, \eqref{eq-eigen-zero} readily follows. 
\end{enumerate}
\end{proof}

\begin{lem}   
\label{lem-bigger-eigen}
There exist uniform (in time $t$ and $\e$) constants $C >0$ and $ \e_0>0$, such that 
for every $0 < \e \leq \e_0$,  if $\lambda_2 \leq  0$ at $P$, then 
$$\lambda_1 \leq C.$$
\end{lem}

\begin{proof} Since $\lambda_2 \leq 0$, we have $G/H \leq 0$. Hence, from \eqref{eq-better-H}  and the bound  $|\langle F_\e, \nu\rangle| \leq C_0$,
for a uniform in time constant $C_0$, we conclude that
$$H  \leq C +  C\,\e  H$$
for a constant $C$ that depends only on the initial data. 
We conclude that for $\e \leq \e_0$, with $\e_0$ sufficiently small depending only on
the initial data $\Sigma_0$, we have 
$$H:= \lambda_1 + \lambda_2 \leq C$$
from which the desired bound on $\lambda_1$ follows with  the aid of the previous
lemma. 
\end{proof}

\begin{lem} 
\label{lem-smaller-eigen}
There exist uniform (in $t$ and $\e$) constants $C >0$ and $ \e_0>0$, such that 
for every $0 < \e \leq \e_0$ we have   
$$\lambda_2 \geq -C.$$
\end{lem}
\begin{proof} Assume that $\lambda_2 < 0$ (otherwise the bound is obvious).
Then, $\lambda_1 >0$ (since $H=\lambda_1+\lambda_2 >0$)  and  by Lemma \ref{lem-bound-F}, we have
$$\kappa_\e:= \frac {G}{H} + \e \, H  \geq - C$$
for a uniform in time constant $C$. Also, by the previous lemma $H \leq \lambda_1 \leq C$.  Hence,
$$|\lambda_2| \leq C\,(1+\e)\,   \frac{\lambda_1+\lambda_2}{\lambda_1} \leq \tilde C.$$
\end{proof}

\begin{remark}
Lemma  \ref{lem-bigger-eigen} implies  that if  the flow terminates because of the blowing up of the second fundamental
form, that could only happen in the   convex region of $\Sigma_t^\e$  where $\lambda_1 \geq 0$, $\lambda_2 \geq 0$. 

\end{remark}

\begin{lem}   
\label{lem-lambda-pinch}
There exist  uniform (in time $t$ and $\e$) constants $C >0$, $C_0 >0$ and $\e_0$, such that  for every   $\e \geq \e_0$ if  $\lambda_1 \geq C_0$ at $P$, then $\lambda_2 >0$ at $P$ and 
$$1 \leq \frac{\lambda_1}{\lambda_2}  \leq C.$$
\end{lem}

\begin{proof} From the previous lemma, $\lambda_2 >0$ if $C_0$ is chosen
sufficiently large. Hence, from the bound \eqref{eq-better-H} we conclude
\begin{equation}\label{eqn-lala}
\begin{split}
(\lambda_1 + \lambda_2)^2 &\leq C_1 \, (\lambda_1 + \lambda_2) + 2T_\e \, [ \,  \lambda_1\, \lambda_2 + \e \, (\lambda_1+\lambda_2)\, ] \\
&\le  \tilde C_1 \, (\lambda_1 + \lambda_2) + \tilde C_2  \, \lambda_1\, \lambda_2
 \end{split}
 \end{equation}
 for some uniform in $\e$ and $t$ constants $\tilde C_1$ and $\tilde C_2$. 
By taking $C_0$ sufficiently large, we can make   
$$(\lambda_1 + \lambda_2)^2 - C_1 \, (\lambda_1 + \lambda_2)  \geq \frac 12 \, (\lambda_1 + \lambda_2)^2.$$  Hence, \eqref{eqn-lala} implies the bound
$$ \lambda_1^2 + \lambda_2^2 \leq 2 \, \tilde C_2 \,   \lambda_1\, \lambda_2  $$
from which the desired estimate  readily follows. 
\end{proof}

To facilitate future references we combine the previous three lemmas in the following
proposition:

\begin{prop}
\label{lem-pinch-e}
There exist $\e_0 > 0$ and  positive constants $C_1, C_2$, uniform in $0<\e<\e_0$ and $t$, so that
for every $0 < \epsilon < \epsilon_0$,  we have
\begin{enumerate}[i.]
\item
$\lambda_2  \ge -C_1$, and 
\item
$\lambda_1 \le C_1\lambda_2 + C_2$. 
\end{enumerate}
\end{prop}

In  the proof of Theorem \ref{thm-e} we will also need the following bound. 

\begin{lem}
\label{lem-unif-H}
There is a uniform constant $C$, independent of $\e$ and $t$ so that 
$$\int_{\Sigma_t^{\e}} H^2\, d\mu_t \le C.$$
\end{lem}

\begin{proof}
We begin by noticing that  that $\int_{\Sigma_t^{\e}} G\, d\mu_t$ is a topological invariant, equal to $2\pi\chi$, where $\chi$ is the Euler charactersistic of $\Sigma_0$. 
Since $\chi=2$ we then have
\begin{equation}\label{eqn-euler}
\int_{\Sigma_t^{\e}} G\, d\mu_t = 4\, \pi.
\end{equation}
At any point we can choose the coordinates in which the second fundamental form is diagonal, with eigenvalues $\lambda_1$ and $\lambda_2$ as before and  $\lambda_1 \ge \lambda_2$. By Lemma \ref{lem-pinch-e} we have $\lambda_1 \leq C_1\, \lambda_2 + C_2$ which gives the inequality  
$$G := \lambda_1\, \lambda_2 \geq \frac 1{C_1}\, \lambda_1^2  - \frac {C_2}{C_1} \, \lambda_1.$$
Using Cauchy-Scwartz we conclude the bound  
$$\lambda_1\, \lambda_2 \geq \tilde{C}_1\lambda_1 ^2 - \tilde{C}_2$$
where $\tilde{C}_1, \tilde{C}_2$ are some uniform constants independent of $\e$ and time.
This yields to the estimate 
$$|A|^2 = \lambda_1^2 + \lambda_2^2 \le C_1\, G + C_2$$
which after  integrated   over $\Sigma_t^{\e}$ implies the bound
$$\int_{\Sigma_t^{\e}} |A|^2\, d\mu_t \le C_1\int_{\Sigma_t^{\e}} G\, d\mu_t + C_2 \, \mu_t(\Sigma_t^\e)$$
with $\mu_t(\Sigma_t^\e)$ denoting, as above,  the surface  area of $\Sigma_t^\e$. By \eqref{eqn-vol}, $\mu_t(\Sigma_t^\e)\leq \mu_0(\Sigma_0)$, where $\mu_0(\Sigma_0)$ denotes  the surface area of $\Sigma_0$. Hence, the lemma readily follows
from \eqref{eqn-euler}.

\end{proof}

\subsection{The proof of Theorem \ref{thm-e}}
Having all the ingredients from the previous sections we will  finish the proof of Theorem \ref{thm-e}. 

\begin{proof}[Proof of Theorem \ref{thm-e}] 
Fix an $\e$ and let $T=T_\e$ be a maximal time up to which the flow exists. 
To simplify the notation we will omit the $\e$-scripts from everything, 
including $T_\e$ and the surface $\Sigma_t^\e$, denoting them by $T$ and $\Sigma_t$ respectively. 
Because of   Proposition  \ref{thm-max-time},  the second fundamental form blows up at time $T$. 
Hence,  there is a sequence of $t_i\to T$ and $p_i\in \Sigma_{t_i}$ so that
$$Q_i:= |A|(p_i,t_i)= \max_{t\in [0,t_i]}\max_{\Sigma_{t_i}}|A|(\cdot,t_i)  \to \infty,
\qquad \mbox{as}\,\, i\to\infty.$$
Consider the sequence $\tilde \Sigma^i_t$ of rescaled solutions defined by 
\begin{equation}
\label{eqn-rescaled}
\tilde{F}_i(\cdot,t) := Q_i(F(\cdot,t_i+\frac{t}{Q_i^2}) - p_i).
\end{equation}
Notice that under the above rescaling all points $p_i$ are  shifted to the origin.
If $g, H$ and  $A :=\{h_{jk}\}$ are the induced metric, the mean curvature and the second fundamental
form of $\Sigma_t$, respectively, then the corresponding rescaled quantities are given by 
$$\tilde{g}_i = Q_i^2 g, \qquad  \tilde{H}_i = \frac{H}{Q_i}, \qquad  |\tilde{A}_i|^2 = \frac{|A|^2}{Q_i^2}.$$
Consider a sequence of rescaled solutions $\tilde{\Sigma}_t^i$. They have a property that 
$$\max_{\tilde{\Sigma}_t^i}|\tilde{A}_i| \le 1, \,\,\, \mbox{for}\,\, t\in [-1,0]  \quad  \mbox{and}
\quad  |\tilde{A}_i|(0,0) = 1.$$ 
The above  uniform estimates on the second fundamental form yield uniform higher
order estimates on $\tilde{F}_i(\cdot,t)$ and the Theorem of
Arzela-Ascoli gives us a uniformly convergent subsequence $\tilde{F}_{i_k}(\cdot,t)$ on compact
subsets, converging to  a smooth  $\tilde{F}(\cdot,t)$ for every $t \in [0,1]$. Notice that 
$$\tilde{\kappa}_i =  \frac{\tilde{G}_i}{\tilde{H}_i} + \e\, {\tilde{H}_i} = \frac{\lambda_1\lambda_2}{Q_i(\lambda_1 + \lambda_2)}+ \e\, \frac{\lambda_1+\lambda_2}{Q_i}$$
and therefore by Proposition \ref{lem-pinch-e}, 
\begin{equation}  
|\tilde{\kappa}_i| \le
\begin{cases}
\frac{C}{Q_i} \quad &\mbox{if} \,\,\, \lambda_1, \lambda_2 << Q_i \\
C, \quad &\mbox{if} \,\,\, \lambda_1, \lambda_2 \sim Q_i
\end{cases}
\end{equation} 
since $\lambda_2 \ge -C$ and $\lambda_1$ is big, comparable to the rescaling constant $Q_i$, if and only if $\lambda_2$ is big and comparable to $Q_i$ (both $\lambda_1$ and $\lambda_2$ are computed at time $t_i + t/Q_i^2$).
This implies  that $\tilde{F}(\cdot,t)$ solves $\frac{\partial}{\partial t}\tilde{F}(\cdot,t) = -\tilde{\kappa_\e}\, \nu$, where
\begin{equation}
\tilde{\kappa}_\e = 
\begin{cases}
0, \quad &\mbox{if} \,\,\,\tilde{\lambda}_1 = 0, \,\, \tilde{\lambda}_2 = 0 \\
\frac{\tilde{\lambda}_1\tilde{\lambda}_2}{\tilde{\lambda}_1+\tilde{\lambda}_2}+\e\, (\tilde \lambda_1+\tilde \lambda_2), \quad 
&\mbox{if}  \,\, \tilde{\lambda}_1 > 0, \,  \tilde{\lambda}_2 > 0.
\end{cases}
\end{equation}  
By Proposition \ref{lem-pinch-e} there are uniform constants
$C_1, C_2$ so that
$$\lambda_1 \le C_1\lambda_2 + C_2$$
which holds uniformly on $\Sigma_t$, for all $t \ge 0$ for which the flow exists, which
after rescaling yields
\begin{equation}
\label{eq-pinch-resc}
\tilde{\lambda}^i_1 \le C_1\tilde{\lambda}^i_2 + \frac{C_2}{Q_i}.
\end{equation}
The previous estimate implies that  the limiting surface (which we denote by 
$\tilde{\Sigma}_0$)  is  convex (possibly not strictly convex).  There are two possibilities for
$\tilde{\Sigma}_0$:  either it is a flat plane or it is a non-flat  complete
weakly convex smooth hypersurface in $\mathbb{R}^3$. Let $\tilde{F}_0$
be a smooth embedding of $\tilde{\Sigma}_0$ into $\mathbb{R}^3$.  Due to
our rescaling, the norm of the second fundamental form of rescaled
surfaces is $1$ at the origin and therefore $\tilde{\Sigma}_0$ is not a
plane, but is strictly convex at least somewhere. It has the  property that
$$\sup_{\tilde{\Sigma}_0}|\tilde{A}| \le C.$$
By the results in \cite{EH} there is a smooth complete solution $\bar{\Sigma}_t$  to the mean
curvature flow
\begin{equation}
\label{eq-mcf}
\begin{cases}
\frac{\partial}{\partial t}\bar{F}(p,t) &= - \bar H\nu(p,t), \qquad p\in \bar{\Sigma}_t, \,\,\, t > 0 \\
\bar{F}(p,0) &= \tilde{F}_0.
\end{cases}
\end{equation} 
The  results in \cite{EH} (see Theorems $2.1$, $2.3$, $3.1$ and $3.4$, which 
provide with  curvature estimates and are  of local nature) imply that   the curvature of $\bar{\Sigma}_t$ stays uniformly bounded for some short time $t\in [0,T_0)$. The evolution for $\bar H$ along the mean curvature flow is given by
$$\frac{\partial}{\partial t} \bar H = \Delta \bar H + \bar H \, |\bar A|^2.$$
 As in \cite{EH}, due to the curvature bounds, the mean curvature $\bar H$ satisfies the conditions of Theorem $4.3$ in \cite{EH} (the maximum principle for parabolic equations on complete hypersurfaces) and therefore nonnegative mean curvature is preserved along the flow. 
This together with the strong maximum principle implies that if $\bar H$ is not identically zero at $t = 0$,  then it becomes strictly positive at $t > 0$. We also know that $\bar{\Sigma}_0$ satisfies
$\bar{\lambda}_1 \le C\, \bar{\lambda}_2$ for a uniform constant $C$, which follows from (\ref{eq-pinch-resc}) after taking the limit as $i\to\infty$. Since we are assuming $\bar{\lambda}_1 \geq \bar \lambda_2$  this can be written as
\begin{equation}
\label{eq-pinch-est0}
\bar h_{ij} \ge \eta \, \bar H\, \bar g_{ij},
\end{equation}
for some uniform constant $\eta > 0$ and we will say the second fundamental form of $\bar{\Sigma}$ is $\eta$-pinched.  

By the curvature bounds, the maximum principle for complete hypersurfaces and the evolution for $\bar h_{ij} - \eta \bar H \bar g_{ij}$ it follows that the pinching estimate
(\ref{eq-pinch-est0}) is preserved by the mean curvature flow (as in \cite{Hu}). In particular, this implies  that $\bar h_{ij}$ is strictly positive definite, which means $\bar{\Sigma}_t$ is strictly convex for $t > 0$. The    result of R. Hamilton  in \cite{Ha2}   states that  a smooth  strictly convex and complete hypersurface with
its second fundamental form $\eta$-pinched must be compact. Hence,  it follows that $\bar{\Sigma}_t$ has to be compact for $t > 0$. In this case, the initial data   $\tilde{\Sigma}_0$ has to be compact as well. 

We recall that $\tilde \Sigma_0$ is the limit of the hyper-surfaces 
$\tilde \Sigma_0^i$ which are obtained  via  re-scaling from the surfaces $\Sigma_{t_i}$. Hence, since $\tilde \Sigma_0$ is compact, there are constants 
$i_0, C$ so that for $i \ge i_0$, we have 
\begin{equation}
\label{eq-diam-shrink}
\diam(\Sigma_{t_i}) < \frac{C}{Q_i} \to 0 \,\,\, \mbox{as} \,\,\, i\to\infty,
\end{equation}
and therefore $\Sigma_{t_i} \to \{\bar{p}\}$.

\begin{claim}
\label{claim-shrink}
For any point $q\in \mathbb{R}^3$, we have
$$\frac{\partial}{\partial t}|F - q|^2 = \mathcal{L_\e}(|F - q|^2) - 2\frac{|A|^2}{H^2}.$$
\end{claim}

\begin{proof}
Follows by a simple computation.
\end{proof}

By Claim \ref{claim-shrink}, $|F - \tilde{p}|_{\max}(t)$ is decreasing along (\ref{equation-reg}) and therefore
$$\Sigma_t  \to \{\tilde{p}\},  \qquad  \mbox{as} \,\,\, t\to T $$
which implies that the surface $\Sigma_t$ shrinks to a point as   as $t\to T$. 
Hence, $\mu_t(\Sigma_t) \to 0$ as $t\to T$. 
It follows by (\ref{eqn-vol}) that  $T$ must be given by  (\ref{eq-ext-time}). 
%
\end{proof}

\section{Passing to the limit  $\e \to 0$}\label{section-lte}
 We will assume in this section that $\Sigma^\e_t$ 
are solutions of the flow \eqref{equation-reg} which satisfy the 
condition \eqref{eq-assump}   uniformly in $\e$,  with $T_\e$  given by
\eqref{eq-ext-time}. We shall show that we can pass to the limit $\e \to 0$ to obtain a 
solution of the \eqref{equation-HMCF} which is defined up to time 
$$T:=\lim_{\e\to 0} T_{\e} = \frac{\mu_0(\Sigma_0)}{4\pi}.$$

The key result is the following uniform bound on the second fundamental form $A$ of $\Sigma_\e$.

\begin{prop}
\label{prop-uniform-e}
Under assumption (\ref{eq-assump}), for any $\tau < T$, there is a uniform constant $C = C(\tau)$ so that
\begin{equation}
\label{eq-unif-e0}
\max_{\Sigma_t^{\e}}|A|(\cdot,t) \le C, \qquad \forall  \e > 0 \quad  \mbox{and} \quad \forall t\in [0,\tau]. 
\end{equation}
where $A$ denotes the second fundamental form of the surface $ \Sigma_t^{\e}$.
\end{prop}  

\begin{proof}
Assume there is $\tau < T$ for which (\ref{eq-unif-e0}) doesn't hold. 
Then,   there exist sequences $t_i\to\tau$, $\e_i \to 0$ and  $p_i\in \Sigma_{t_i}^{\e_i}$ so that
$$Q_i := |A|(p_i,t_i) = \max_{\Sigma_t^{\e_i}\times [0,t_i]}|A| \to \infty \,\,\, \mbox{as} \,\,\, j\to\infty.$$  
Consider,   as before,  the   rescaled sequence of solutions $\tilde \Sigma^i_t$ defined
by the immersions $\tilde{F}_i(\cdot,t): M^2 \to \R^3$, 
$$\tilde{F}_i(\cdot,t) := Q_i(F_{\e_i}(\cdot,t_i+\frac{t}{Q_i^2}) - p_i).$$
Due to our rescaling, the second fundamental form of rescaled surfaces is uniformly bounded in $i$.
This uniform estimates on the second fundamental form yield uniform $C^2$-bounds on $\tilde{F}_i(\cdot,0)$
and the Theorem of Arzela-Ascoli gives us a uniformly convergent subsequence on compact subsets,
converging in the in $C^{1,1}$-topology  to a $C^{1,1}$ surface $\tilde \Sigma$ defined by the immersion  $\tilde{F}$. 

By Lemma \ref{lem-pinch-e},  there are uniform constants
$C_1, C_2$ so that the estimate
$$\lambda_1 \le C_1\lambda_2 + C_2$$
holds uniformly on $\Sigma_t^{\e}$, for all $t \ge 0$ for which the flow exists,  and all $\e$, which
after rescaling yields to the estimate
$$\tilde{\lambda}^i_1 \le C_1\tilde{\lambda}^i_2 + \frac{C_2}{Q_i}.$$ 
Hence,  the limiting surface 
$\tilde{\Sigma}$ is  convex.  There are two possibilities for
$\tilde{\Sigma}$, either it is a flat plane,  or it is a complete  convex  $C^{1,1}$-hypersurface.  

Due to our rescaling, the curvatures of the  rescaled surfaces $\tilde \Sigma^i_t$ are uniformly bounded in $i$. This in particular  implies  a uniform local Lipshitz
condition on $\tilde{F}_i(M^2,0)$.  This means that  there are fixed
numbers $r_0$ and $C_0$ so that for every $q\in\tilde{F}_i(M^2)$,
$\tilde{F}_i(U_{r_0,q})$ (where $U_{r_0,q}$ is a component of
$\tilde{F}_i^{-1}(B_{r_0}(\tilde{F}_i(q)))$ containing $q$,  and
$B_{r_0}$ is a ball of radius $r_0$ in $\mathbb{R}^3$) can be written
as the graph of a Lipshitz function over a hyperplane in
$\mathbb{R}^3$ through $\tilde{F}_i(q)$ with Lipshitz constant less
than $C_0$. Notice that both $C_0$ and $r_0$ are independent of $i$,
they both depend on a uniform upper bound on the second fundamental
form. This means the limiting surface $\tilde{\Sigma}$ will satisfy a uniform local
Lipshitz condition.

\begin{lem}
\label{lem-plane}
The limiting hypersurface $\tilde{\Sigma}$ is not a  plane.
\end{lem}

\begin{proof}
Assume that the limiting hypersurface $\tilde{\Sigma}$ is a plane. 
Then, for  each  $i$ we can write $\tilde \Sigma_i$   in a  neighbourhood which is a ball 
$B(0,1)$ of radius $1$ around the origin  as a graph of a 
$C^2$-function $\tilde u_i$, over some hyperplane $\mathcal{H}_i$. In particular, we can choose one that is tangent to $\tilde{\Sigma}_i$
at the origin. Then
\begin{equation}
\label{eq-graph}
\tilde{h}^i_{jk} = \frac{D_{jk} \tilde u_i}{(1+|D\tilde u_i|^2)^{\frac{1}{2}}}.
\end{equation}
We can choose a coordinate system in each hyperplane so that the
second fundamental form and also $D^2 u_i$ are diagonal at the origin.
The function $u_i$ is a  height function that measures the distance of  our
surface from the hyperplane $\mathcal{H}_i$. We also have that
$u_i\stackrel{C^{1,1}}{\to}\tilde{u}$ as $i\to\infty$ and $u_i(0) = 0$
for all $i$. If $\tilde{\Sigma}$ were a plane then $\tilde{u} \equiv
0$ and $|D\tilde{u}_i| \equiv 0$ which would imply
$|\tilde{u}|_{C^{1,1}} \equiv 0$. Take $\e > 0$ very small.  Then
there would exist $i_0$ so that for $i \ge i_0$, $|u_i|_{C^{1,1}} <
\e$ on $B(0,1) \subset \tilde{\Sigma}_i$. Since we have (\ref{eq-graph}),
the last estimate would contradict the fact $|\tilde{A}_i|(0,0) = 1$,
that is valid by the way we rescaled our solution.
\end{proof}

It follows from the previous lemma and the discussion above that    $\tilde{\Sigma}$ is a complete 
convex, non-flat  $C^{1,1}$-surface that satisfies $\tilde{\lambda}_1
\le C\tilde{\lambda}_2$, whenever those quantities are defined (since
a surface is $C^{1,1}$, the principal curvatures are defined almost
everywhere).  Because of our uniform curvature estimates of the
rescaled sequence, $\tilde{\Sigma}$ is a uniformly locally Lipschitz
surface. By the results in \cite{EH} there is a solution $\bar F_t$   of the  Mean Curvature flow  (\ref{eq-mcf}) with initial data $\tilde \Sigma$  on some time interval $[0,T_1)$ and $\bar {F}_t$
is smooth for $t > 0$. We can now carry out the same argument as in
the proof of Theorem \ref{thm-e} to show that $\tilde{\Sigma}$ has to
be compact. That would mean that for $j >>1$, 
\begin{equation}
\label{eq-ar-small}
\diam\, (\Sigma_{t_j}^{\e_j}) \le \frac{C}{Q_j} \quad \mbox{and} \quad \ar\, (\Sigma_{t_j}^{\e_j}) \le \frac{C}{Q_j}
\end{equation}
for  a uniform constant $C$. Since $T_j \to \tau < T$, \eqref{eq-ar-small} and  Lemma \ref{lem-unif-H} contradict  \eqref{eqn-vol}. This shows that \eqref{eq-unif-e0} holds true, therefore finishing our proof. 
\end{proof}

We will now show that because of \eqref{eq-unif-e0} we can pass along subsequences $\e_i \to 0$ and show that the solutions $\Sigma^{\e_i}_t$ converge to a solution $\Sigma_t$
of \eqref{equation-HMCF}.

Observe first that since  ${\partial F_{\e}}/{\partial t} = -(\kappa + \e\,  H)\nu$, 
 by Proposition \ref{prop-uniform-e}  we have that $|{\partial F_{\e}}/{\partial t}| \le C$, uniformly in $\e$.   Hence,   $F_{\e}$ is uniformly Lipshitz  in $t$. Combining this with  Proposition \ref{prop-uniform-e} and the assumption \eqref{eq-assump}, we conclude    that for every $\tau < T$ there is a subsequence $\e_i\to 0$ and a $1$-parameter family of $C^{1,1}$
surfaces $F(\cdot,t)$,   so that $F_{\e_i} \to F$ in the $C^{1,1}$ norm,  ${\partial F_{\e_i}}/{\partial t}\to {\partial F}/{\partial t}$ in the  weak sense and $F$ satisfies 
\begin{equation}
\label{eq-our-sol}
\frac{\partial F}{\partial t} = -\kappa\, \nu.
\end{equation} 
Due to (\ref{eq-assump}) our solution has the property that 
\begin{equation}\label{eqn-delta}
\ei_{\Sigma_t \times [0,T)} H\ge \delta.
\end{equation}

\begin{claim} The limiting  solution of  (\ref{eq-our-sol}) does not depend on the  sequence $\e_i \to 0$. 
\end{claim}

\begin{proof}
Consider the evolution of a surface $\Sigma_t$ by a fully-nonlinear  equation of the form 
\begin{equation}
\label{eq-F}
\frac{\partial F}{\partial t}  = - \mathcal{F}(h_{ij})\, \nu
\end{equation}
where $h_{ij}$ is the second fundamental form and $\mathcal{F}$ is a function of
the eigenvalues of $\{h_{ij}\}$,  which we denote by $\lambda_1,\lambda_2$ 
and assume that $\lambda_1 \geq \lambda_2$.  Let $\mu = \lambda_2/{\lambda_1}$ and 
take 
\begin{equation}
\mathcal{F}(\lambda_2,\mu) =
\begin{cases}
\frac{\lambda_1\lambda_2}{\lambda_1+\lambda_2} =
 \frac{\lambda_2}{1+\mu}, & \text{for $\mu\ge -\delta_1$} \\
\frac{\lambda_2}{1-\delta_1}, & \text{otherwise}
 \end{cases}
 \end{equation}
 which we can be written as
 \begin{equation}
 \label{eq-other-form}
\mathcal{F}(h_{ij}) =
\begin{cases}
\kappa, & \text{for $Hg_{ij} \ge (1-\delta_1)h_{ij}$} \\
\frac{H + \sqrt{H^2 - 4G}}{1-\delta_1}, & \text{otherwise.}
\end{cases}
\end{equation}
We can also consider  solutions of  (\ref{hmcf0}) in the viscosity sense (defined in \cite{CGG} and \cite{ES}). In that case  (\ref{eq-F}) can be written in the form
\begin{equation}
\label{eq-visc}
u_t  =
\begin{cases}
\frac{\det(D_i(\frac{D_j u}{|Du|}))}{\dv(\frac{Du}{|Du|})}, \,\,\, \text{when $\dv(\frac{Du}{|Du|})\delta_{ij} \ge (1-\delta_1)D_i(\frac{D_j u}{|Du|})$} \\
\frac{\dv(\frac{Du}{|Du|}) + \sqrt{\dv(\frac{Du}{|Du|}) - 4\det(D_i(\frac{D_ju}{|Du|}))}}{1-\delta_1} \qquad  \text{otherwise}.
\end{cases}
\end{equation}
Equation (\ref{eq-visc}) can be expressed as 
\begin{equation}
\label{eq-visc1}
u_t + \mathcal{F}_1(t,\nabla u, \nabla^2 u) = 0,
\end{equation}
with
$$\mathcal{F}_1(t,p,X) = -\frac{|p|\det(X - \frac{p}{|p|}\otimes (X\cdot \frac{p}{|p|})}{\tr((I - \frac{p}{|p|}\otimes \frac{p}{|p|})\cdot X)}$$
if
$$ I \cdot  \tr((I - \frac{p}{|p|}\otimes \frac{p}{|p|})\cdot X)  \ge (1-\delta_1)(X - X\cdot \frac{p}{|p|}\otimes \frac{p}{|p|})$$
and

\begin{equation*}
\begin{split}
\mathcal{F}_1(t,p,X) =
&\frac{1}{1-\delta_1}\left (\frac{1}{|p|}(\tr((\delta_{ij} - \frac{p}{|p|}\otimes \frac{p}{|p|})X \right ) \\
\\ &+ 
\sqrt{\frac{1}{|p|^2}[(\tr((\delta_{ij} - \frac{p}{|p|}\otimes \frac{p}{|p|})X))^2 - \frac{4}{|p|}\det(X - \frac{(X\cdot p)}{|p|}\otimes\frac{p}{|p|})})
\end{split}
\end{equation*}
otherwise. 

Notice that the lower bound \eqref{eqn-delta}  together with our curvature pinching estimates (that follow from the Proposition \ref{lem-pinch-e}) imply that 
$$H\, g_{ij} \ge (1-\delta_1)h_{ij}$$
 for some $1 >\delta_1 > 0$. This implies that we can view a solution to (\ref{eq-our-sol}) as a solution to (\ref{eq-F}) with $F$ as in (\ref{eq-other-form}). The function $\mathcal{F}_1(t,p,X)$ is continuous on $(0,T) \times \mathbb{R}^2\backslash \{0\} \times S^{2\times 2}$, it satisfies the conditions of Theorem $7.1$ in \cite{CGG}  and (\ref{eq-visc1}) is a degenerate parabolic geometric equation in the  sense of Definition $5.1$ in \cite{CGG}. Theorem $7.1$ in \cite{CGG} shows the uniqueness of viscosity solutions to  (\ref{eq-visc1}). The $C^{1,1}$ solution on $[0,T)$ constructed above is a viscosity solution to (\ref{eq-visc1}) and by the uniqueness result it is the  unique $C^{1,1}$ solution to (\ref{eq-our-sol}). This means that the limiting  solution of  (\ref{eq-our-sol}) does not depend on the sequence $\e_i \to 0$. 
 \end{proof}


\section{Radial case}

In this section we will employ the results from the previous section to completely describe
the long time behaviour of (\ref{hmcf0}) in the case of surfaces of revolution, $r = f(x,t)$
around the $x$-axis. For such a surface of revolution the two principal curvatures are 
given by
\begin{equation}
\label{eq-princ-cur}
\lambda_1 = \frac{1}{f(1+f_x^2)^{\frac{1}{2}}} \qquad \mbox{and} \qquad 
\lambda_2 = -\frac{f_{xx}}{(1+f_x^2)^{\frac{3}{2}}}.
\end{equation}
Therefore,
$$H = \lambda_1 + \lambda_2 = \frac{-ff_{xx} + f_x^2 + 1}{f(1+f_x^2)^{\frac{3}{2}}} > 0$$
and
$$G = \lambda_1\, \lambda_2 = \frac{-f_{xx}}{f(1+f_x^2)^2}.$$
When the surface evolves by (\ref{hmcf0}), $f(x,t)$ evolves by
\begin{equation}
\label{eq-hmcf-rad}
f_t = \frac{f_{xx}}{-ff_{xx} + f_x^2 + 1}.
\end{equation} 
We will consider solutions $f(\cdot,t)$ on an interval  $I_t = [a_t,b_t] \subset [0,1]$ such that $f(a_t,t) = f(b_t,t) = 0$,
$f > 0$ and $\tilde{H} = -ff_{xx} + f_x^2 + 1 > 0$. From (\ref{eq-princ-cur}) we see that
$\lambda_1 > 0$ and $\lambda_2$ changes its sign,  depending on the convexity of $f$.
The linearization of (\ref{eq-hmcf-rad}) around a point $f$ is
\begin{equation}
\label{eq-linear-rad}
\tilde{f}_t = \frac{1+f_x^2}{\tilde{H}^2}\, \tilde{f}_{xx} - \frac{2f_xf_{xx}}{\tilde{H}^2}\, \tilde{f}_x
+ \frac{f_{xx}^2}{\tilde{H}^2}\tilde{f}
\end{equation}
which is uniformly parabolic  when $\tilde{H}$ is away from zero, no matter what is the sign of 
the smaller eigenvalue $\lambda_2$. 

\begin{thm}
\label{thm-radial}
Assume that at time $t = 0$, $\Sigma_0$ is a $C^{1,1}$ star-shaped surface of revolution $r = f(x,0)$,
for $x\in [0,1]$,  $f(0,0) = f(1,0) = 0$, $f(\cdot,0) > 0$ and $H > 0$. Then, the flow exists
up to the maximal  time 
$$T= \frac {\mu_0(\Sigma_0)}{4\pi}$$
when the surface $\Sigma_t$ contracts  to a point. Moreover, the surface becomes 
strictly convex at time $t_1 < T$ and  asymptotically spherical at its extinction time $T$.

\end{thm}

Since the equation is strictly parabolic when $\tilde H >0$, the short time existence of a smooth solution on some time interval $[0,\tau]$, 
follows by classical results.  Having a smooth
solution to (\ref{hmcf0}) on $[0,\tau]$  implies that we have  a smooth solution
$f(\cdot,t)$ to (\ref{eq-hmcf-rad}). By the comparison principle, $f(x,t)$ is
defined on $I_t = [a_t,b_t] \subset [0,1]$ and $f(a_t,t) = f(b_t,t)
= 0$. Since the surface is smooth and  $H >0$ on $[0,\tau]$, the   
expressions for $\lambda_1$ and $\lambda_2$  in (\ref{eq-princ-cur}) yield to the bounds
$$\limsup_{x\to a_t} f\,  |f_x| \le C_1(t) \quad \mbox{and} \quad  \limsup_{x\to
b_t}f\, |f_x| \le C_2(t), \quad \mbox{for}\,\, 0 \leq t \leq \tau.$$   
In the next lemma we will show that the above bounds do not depend on the lower
bound on $H$, but only on the initial data. 
%

\begin{lem}\label{lem-ff}
Assume that the  solution $f$ is smooth on $[0,t_0)$,  for some $t_0 \leq T$ and $H >0$
on $[0,t_0)$. Then, there exists a uniform constant $C$, depending only on initial data, so that
\begin{equation}
\label{eq-upper-ffx}
f^2f_x^2  \le C, \qquad \mbox{for all} \,\, t \in [0,t_0).
\end{equation}  
\end{lem}
\begin{proof}
We will bound $f^2 \,  f_x^2$ from above  by the maximum principle.  Let us compute its evolution equation. We first compute the evolution of  $f_x$ by differentiating 
(\ref{eq-hmcf-rad}) in $x$. We   get
\begin{equation}
\label{eq-der-der}
(f_x)_t = \frac{f_{xxx} \, (1+f_x^2) - f_x\,  f_{xx}^2}{\tilde{H}^2},
\end{equation}
which yields the following equation for $f_x^2$:
\begin{eqnarray*}
(f_x^2)_t &=& \frac{2f_{xxx}f_x\, (1+f_x^2) - 2f_x^2\,  f_{xx}^2}{\tilde{H}^2} =
\frac{((f_x^2)_{xx} - 2f_{xx}^2)\, (1+f_x^2) - 2f_x^2f_{xx}^2}{\tilde{H}^2} \\
&=& \frac{(f_x^2)_{xx}\, (1+f_x^2) - 4f_x^2f_{xx}^2 - 2f_{xx}^2}{\tilde{H}^2}.
\end{eqnarray*}
The function $f^2$ satisfies the equation 
$$( f^2)_t  = \frac{(f^2)_{xx} - 2f_x^2}{\tilde{H}^2}.$$
Combining the last  two equations we obtain 

\begin{eqnarray*}
(f^2  f_x^2)_t &=&
\frac{(f_x^2)_{xx}\, (1+f_x^2) - 4f_x^2f_{xx}^2 - 2f_{xx}^2}
{\tilde{H}^2}\, f^2 + 2\, \frac{f_{xx} f f_x^2}{\tilde{H}^2}\\
&=&\frac{(f^2f_x^2)_{xx}\,(1+f_x^2)}{\tilde{H}^2} - \frac{(1+
f_x^2)((f^2)_{xx}\,  f_x^2 - 2(f_x^2)_x\, (f^2)_x)}
{\tilde{H}^2} + 2\, \frac{f_{xx}ff_x^2}{\tilde{H}^2}. 
\end{eqnarray*} 
Let $t < t_0$. We distinguish the following two cases:

\noindent{\em Case 1.} The $(f^2f_x^2)_{\max}(t)$ is attained in the interior of $(a_t,b_t)$. Then,
 at that point $(f^2f_x^2)_x = 0$, which implies (since
$f(\cdot,t) > 0$ in the interior) that
\begin{equation}
\label{eq-rel}
f_x^3 = -f f_xf_{xx}.
\end{equation}
Hence, the maximum principle implies the differential inequality
\begin{eqnarray}
\label{eq-subst1}
\frac{d}{dt}(f^2f_x^2)_{\max}(t) &\le& -\frac{1+f_x^2}{\tilde{H}^2}\left ((f^2)_{xx}f_x^2 - 2(f_x^2)_x(f^2)_x \right) 
+ 2\, \frac{f_{xx}ff_x^2}{\tilde{H}^2} \nonumber \\
&=&  -8\, \frac{(1+f_x^2)f_x^4}{\tilde{H}^4} - 2\frac{f_x^4}{\tilde{H}^2} \le 0.
\end{eqnarray} 

\smallskip

\noindent{\em Case 2.} The  $(f^2f_x^2)_{\max}(t)$ is attained at
one of the tips $\{a_t, b_t\}$. Assume it is attained at $a_t$. The
point of the surface $\Sigma_t$ that arises from $x=a_t$ can be viewed
as the interior point of $\Sigma_t$ around which our surface is
convex. We can solve locally, around the point $x = a_t$ (say for
$x\in [a_t,x_t]$) the equation $y = f(x,t)$ with respect to $x$,
yielding to the map $x = g(y,t)$. Notice that $f f_x ={y}/{g_y}$
and that $x = a_t$ corresponds to $y = 0$. Since $\{f(x,t)|x\in
[a_t,x_t]\}\cup \{-f(x,t)|x\in [a_t,x_t]\}$ is a smooth curve, we have
that $x = g(y,t)$ is a smooth graph for $y\in [-f(x_t,t),f(x_t,t)]$.
If $f^2f_x^2(\cdot,t)$ attains its maximum somewhere in $[a_t,x_t)$,
then ${y^2}/{g_y^2}$ attains its maximum in the interior of
$(-f(x_t,t),f(x_t,t))$. 

We will now compute the evolution of $y^2/g_y^2$ from the evolution of $f^2\, f_x^2$.
Since
$$f_x(x,t) = \frac{1}{g_y(y,t)}$$ from  the evolution of $f^2\, f_x^2$ we get
$$
\left ( \frac{y^2}{g_y^2} \right )_t = 
\frac{(1+g_y^2)}{g_y^2\tilde{H}^2} \left (\frac{y^2}{g_y^2} \right )_{xx} -
\frac{(1+g_y^2)}{g_y^2\tilde{H}^2} \left ((y^2)_{xx}\, \frac{1}{g_y^2} \
- 2 \left (\frac{1}{g_y^2}\right )_x(y^2)_x \right ) + \frac{2\, y_{xx}y}{g_y^2\tilde{H}^2}.
$$
By direct computation we have
$$\left (\frac{1}{g_y^2} \right )_x = -\frac{2g_{yy}}{g_y^3} \quad \mbox{and} \quad 
(y^2)_x = \frac{2y}{g_y} \quad \mbox{and} \quad  y_{xx} = -\frac{g_{yy}}{g_y^3}$$
and
$$\left (\frac{y^2}{g_y^2} \right )_{xx} = \left (\frac{y^2}{g_y^2} \right )_{yy} - \left (\frac{y^2}{g_y^2} \right )_y \,  \frac{g_{yy}}{g_y^3} \qquad \mbox{and} \qquad  (y^2)_{xx} = \frac{2}{g_y^2} - \frac{2yg_{yy}}{g_y^3}.$$\\
Combining the above yields to 
\begin{eqnarray*}
\left (\frac{y^2}{g_y^2} \right)_t &=&
\frac{(g_y^2+1)}{g_y^2\tilde{H}^2} \left(\frac{y^2}{g_y^2} \right )_{yy} -
\left (\frac{y^2}{g_y^2}\right )_y\,\frac{g_{yy}\, (1+g_y^2)}{g_y^5\tilde{H}^2}\\
&-& \frac{(1+g_y^2)}{g_y^2\tilde{H}^2}\left (\frac{2}{g_y^4} + 2\frac{yg_{yy}}{g_y^5} \right)
- \frac{2\, yg_{yy}}{g_y^5\tilde{H}^2}
\end{eqnarray*}
which can be re-written  it as
\begin{eqnarray}
\label{eq-g1}
\left (\frac{y^2}{g_y^2} \right)_t &=&
\frac{(g_y^2+1)}{g_y^2\tilde{H}^2}\, \left (\frac{y^2}{g_y^2} \right )_{yy} -
\left (\frac{y^2}{g_y^2} \right )_y\, \frac{g_{yy}(1+g_y^2)}{g_y^2\tilde{H}^5}\nonumber\\
&-& \frac{(1+g_y^2)}{g_y^2\tilde{H}^2}\left(\frac{2}{g_y^4} + \frac{2\, y^2g_yg_{yy}}{yg_y^6} \right) - \frac{2\, y^2g_yg_{yy}}{yg_y^6\tilde{H}^2}.
\end{eqnarray}
At the maximum point of ${y^2}/{g^2_y}$ we have 
$$y^2g_yg_{yy} = yg_y^2.$$
This together with the maximum principle applied to (\ref{eq-g1}) yield to the differential inequality
\begin{equation}
\label{eq-subst2}
\frac{d}{dt}\left (\frac{y^2}{g_y^2} \right )_{\max}(t) \le -\frac{4\, (1+g_y^2)}{g_y^6\tilde{H}^2}
 - \frac{2}{\tilde{H}^2g_y^4} \le 0. 
\end{equation}

Estimates (\ref{eq-subst1}) and (\ref{eq-subst2}) imply that $(f^2f_x^2)(x,t) \le C$,
for all $x\in [a_t,b_t]$ and all $t \leq t_0$,  where $C$ is a uniform
constant independent of time. This finishes the proof of the lemma. 
\end{proof}

\begin{cor}
\label{lem-lower-H0}
Let $T= {\mu_0(\Sigma_0)}/{4\pi}$ be as in Theorem \ref{thm-radial}.  Then, there exists a uniform constant $\delta$, depending only on the initial data, so that
$H \ge \delta > 0$, for all $t\in [0,T)$.
\end{cor}

\begin{proof}
It is enough to show that if $H >0$ on $[0,t_0)$, then $H \geq \delta >0$ there. 
We recall that $\lambda_1 = {1}/{f(1+f_x^2)^{{1}/{2}}}$. Hence, the estimate (\ref{eq-upper-ffx}) yields to  the bound 
\begin{equation}
\label{eq-lam1}
\lambda_1 \ge c >0  \qquad \mbox{on} \,\, \Sigma_{t}, \quad \mbox{for}
\,\, t\in [0,t_0).
\end{equation}
Since $H = \lambda_1 + \lambda_2$, if $\lambda_2 \ge 0$, then $H \ge \lambda_1 \ge c$.
If $\lambda_2 < 0$ and  $H \le c/{2}$ (otherwise we are done)  by (\ref{eq-lam1}) we have 
$$\lambda_1 - |\lambda_2| \le \frac c2 \Rightarrow |\lambda_2| \ge \frac{c}{2}.$$
Observe next that Lemma \ref{lem-bound-F}  implies the  bound 
$$\frac{\lambda_1|\lambda_2|}{H} \le C, \qquad \mbox{for a uniform constant} \,\, C.$$
Hence 
$$H \ge \frac{\lambda_1|\lambda_2|}{C} \ge \frac{c^2}{2C}.$$
In any case, we have
$$H \ge \min \left \{\frac{c}{2}, \frac{c^2}{2C} \right \}$$
which shows our lemma with $\delta:= \min  \{\frac{c}{2}, \frac{c^2}{2C}  \}$. 
\end{proof}

\begin{lem}
\label{lem-blowing-up}
Let $[0,T)$ be the maximal interval of existence of  a solution to (\ref{hmcf0}). Then,  
$\max_{\Sigma_t}|A|$ becomes unbounded as $t\to T$.  
\end{lem}

\begin{proof}
Assume that $\sup_{\Sigma_t}|A| \le C$, for 
all $t\in [0,T)$ and write  $$H = \frac{\tilde{H}}{f(1+f_x^2)^{3/2}}$$ with $\tilde H = -ff_{xx} + f_x^2 +1$.  Then $H \leq C$ (since $|A|$ is bounded) and  $H \geq \delta >0$ (by the previous result). Hence, 
$$c_1 \le \frac{f(f_x^2 + 1)^{3/2}}{\tilde{H}} \le c_2$$
which implies
$$ \frac{c_1}{f(1+f_x^2)^{1/2}} \le \frac{1+f_x^2}{\tilde{H}} \le  \frac{c_2}{f(1+f_x^2)^{1/2}}.$$
We can rewrite it as
\begin{equation}
c_1\, \lambda_1 \le \frac{1+f_x^2}{\tilde{H}} \le c_2\, \lambda_1
\end{equation}
which together with (\ref{eq-lam1}) and $|A| \le C$ imply the bounds 
\begin{equation}
\label{eq-unif-ellip}
C_1 \le \frac{1+f_x^2}{\tilde{H}}  \le C_2
\end{equation}
for uniform constants $C_1, C_2$, for all $t\in [0,T)$. This means the linearization
(\ref{eq-linear-rad}) of (\ref{eq-hmcf-rad}) is uniformly elliptic on time interval $[0,T)$.
If our surface of revolution at time $t$ is given by an embedding $F(\Sigma,t)$, which
is a solution to (\ref{hmcf0}), $|A| \le C$ implies $|F|_{C^2} \le C$ on the time interval
$[0,T)$ and the speed $|\kappa| \le C$ (we will use the same symbol $C$ to denote 
different uniform constants).  It is easy to see that $F(\cdot,t)$
converges to a continuous limit $F(\cdot,T)$ as $t\to T$, since
$$|F(x,t_1) - F(x,t_2)| \le \int_{t_1}^{t_2}|\kappa|\, dt \le C|t_1 - t_2|.$$  
Due to 
$$|\frac{\partial}{\partial t}g_{ij}|^2 = |2h_{ij}\kappa|^2 \le 4|A|^2\kappa^2 \le C$$
and \cite{Ha1} we have that $F(\cdot,T)$ represents a surface. It is a
$C^{1,1}$ surface of revolution $r=f(x,T)$ around the $x$-axis that
comes as a limit as $t\to T$ of surfaces of revolution $r=f(x,t)$.
Take $0 < \e << b_T - a_T$ arbitrarily small. Consider $f(r,t)$ on
$x\in [a_t +\e, b_t-\e]$, that is, away from the tips $x = a_t$ and $x
= b_t$ where $f = 0$ and $f_x$ becomes unbounded. Since our solution
is $C^{1,1}$, $c_1 \le f(r,T) \le c_2$ and $|f_x| \le c_3$, at time $t
= T$ and for $x\in [a_T+\e,b_T-\e]$, where $c_1, c_2, c_3$ all depend
on $\e$. Due to (\ref{eq-unif-ellip}), equation (\ref{eq-hmcf-rad}) is
uniformly parabolic and standard parabolic estimates yield
\begin{equation}
\label{eq-away-tips}
|f(\cdot,T)|_{C^k} \le C(\e,k), \qquad \mbox{for every} \,\, k > 0 \,\, \mbox{and} \,\,
x\in [a_T+\e,b_T-\e].
\end{equation} 
We can repeat the previous discussion to every
$\e > 0$ to conclude that our surface $\Sigma_T$ is smooth for $x\in
(a_T,b_T)$. By writing our surface locally as a graph $x = g(y,t)$ around the tips (at which our surface is 
strictly convex), we can show that our surface is smooth at the tips as well (similar methods 
to those discussed above apply in this case).
\end{proof}

The same proof as the one for the flow  (HMCF$_\e$) which was presented in the previous
section, shows that our radial surface $\Sigma_t$ shrinks to a point at $T = \frac{\mu_0(\Sigma_0)}{4\pi}$, where $\mu_0(\Sigma_0)$
is the total area of $\Sigma_0$. In particular, this means $f(x,t) \to 0$ as $t\to T$. 

We will show next that  at some time $t_1  < T$ the surface $\Sigma_{t_1}$ becomes strictly convex. This will follow from the next lemma.

\begin{lem}\label{lem-lower}
Assume that $f$ is a solution of the HMCF on $[0,T)$. Then, there exists a constant $c >0$, 
independent of $t$,  such that 
$f(x,t) \geq c$, at all points $(x,t)$, with $0 \leq t < T$ and   $f_{xx}(x,t) \geq 0$.
\end{lem}

\begin{proof}
Fix $t < T$. Since  our surface $\Sigma_t$ is  convex around the tip $x = a_t$
we have  $f_{xx} \leq 0$ there.  Let  $c_t$ be the largest number in $[a_t,b_t]$
so that $\Sigma_t$ is strictly convex for $x\in [a_t,c_t]$.  If $c_t=b_t$, then $\Sigma_t$ is
 convex and we have nothing to show. Otherwise, 
$f_{xx}(x,t) \leq  0$ for $a_t \leq x \leq  c_t$ and $f_{xx}(x,t) > 0$ 
in   $(c_t,c_t+\e_t)$ for some $\e_t >0$. Hence, 
$f_x(\cdot,t)$ is increasing in $x$,  for $x\in (c_t,c_t+\e_t)$. 

Consider the function $f_x(\cdot,t)$ on the interval $x\in [c_t,b_t)$.
From the above discussion and the fact that  $\lim_{x\to b_t} f_x(x,t) = -\infty$, 
we conclude that the maximum 
$$M(t):= \max \, \{ \, f_x(x,t), \,\, x \in [c_t,b_t]\, \}$$
is attained in the interior of $[c_t,b_t]$.
Recall the evolution equation for $f_x$ to be
$$(f_x)_t = \frac{f_{xxx}(1+f_x^2)}{\tilde{H}^2}
- \frac{f_xf_{xx}^2}{\tilde{H}^2}.$$
Hence, assuming that $M(t) \geq 0$, the  maximum principle 
implies that $M'(t) \leq 0$. This shows that $f_x$ is uniformly bounded from above
on $[c_t,b_t]$. Since a similar argument can be applied near the other tip $b_t$,
we finally conclude that   $|f_x|$ is uniformly bounded in the  non-convex part
(if it exists) away from the tips.

We will now conclude the proof of the lemma. Assume that $f_{xx}(x,t) \ge 0$, which holds in a non-convex part of our evolving surface.  At that point, we have
$$\lambda_2 := -\frac{f_{xx}}{(1+f_x^2)^{3/2}} \le 0.$$
Since  $\lambda_2 \leq  0$, Lemma \ref{lem-bigger-eigen} implies the bound 
$$\lambda_1:=  \frac{1}{f(1+f_x^2)^{1/2}} \leq C$$
which reduces to the  the bound
$$f \ge \frac{1}{C(1+f_x^2)^{1/2}} \ge \frac{1}{\tilde{C}} =: c$$
in the non-convex part where $f_x^2 \leq C$, uniformly in $t$. This finishes the
proof of the lemma.  
\end{proof}

{\em We will now conclude the proof of Theorem \ref{thm-radial}:}  Since $f(x,t) \to 0$ as $t\to T$, with $T=\frac{\mu_0(\Sigma_0)}{4\pi}$, there is some time $t_1 <T$ so that 
$$f(x,t) < \frac c2, \qquad \mbox{for all} \,\, x \in [a_t,b_t]   $$ where $c > 0$ is the  constant taken from Lemma \ref{lem-lower}. Hence,  by   Lemma  \ref{lem-lower}
the surface $\Sigma_t$ is convex for $t \geq t_1$. 
Since $H  \geq \delta >0$ for all $t <T$,  the surface $\Sigma_{t_1}$ is strictly convex. The result of Andrews  in \cite{An1}, implies that $\Sigma_t$ shrinks asymptotically spherically to a point as $t\to T$.

\end{document}